\documentclass[11pt]{amsart}

\usepackage{lmodern}			% Usa a fonte Latin Modern			
\usepackage[T1]{fontenc}		% Selecao de codigos de fonte.
\usepackage{color}				% Controle das cores
\usepackage{graphicx}			% InclusÃ£o de grÃ¡ficos--
\usepackage{epsfig,psfrag}
\usepackage{empheq}
\usepackage{lmodern}			% Usa a fonte Latin Modern			
\usepackage{amsmath,amsthm,amsfonts,amssymb,amscd,amstext}
\usepackage[top=2.0cm,bottom=2.0cm,left=2.0cm,right=2.0cm]{geometry}
\graphicspath{{Figures/}}
\theoremstyle{plain}
\newtheorem{maintheorem}{Theorem}
\newtheorem{teo}{Theorem}[section]
\newtheorem{lemma}[teo]{Lemma}
\newtheorem{remark}[teo]{Remark}
\newtheorem{prop}[teo]{Proposition}
\newtheorem{cor}[teo]{Corolary}

\newtheorem{ex}{Example}
\newtheorem{defi}[teo]{Definition}

\newcommand{\RR}{{\mathbb R}}
\newcommand{\R}{{\mathbb R}}

\newcommand{\N}{\mathbb{N}}

\newcommand{\MF}{\mathfrak{M}(\Phi)}
\newcommand{\Mft}{\mathfrak{M}(\phi^t)}
\newcommand{\Mfone}{\mathfrak{M}(\phi^1)}
\newcommand{\MS}{\mathfrak{M}(S)}

\newcommand{\EF}{\mathfrak{M}_{\mathrm erg}(\Phi)}

\newcommand{\U}{\mathcal{U}}
\newcommand{\V}{\mathcal{V}}
\newcommand{\C}{\mathcal{C}}

\newcommand{\Q}{\mathbb{Q}}

\renewcommand{\epsilon}{\varepsilon}

\newcommand{\Gen}{\lim_{n\rightarrow\infty}\frac{1}{n}\sum_{j=0}^{n-1}\varphi(f^j(x))=\int\varphi{d\mu}}
\newcommand{\OB}{\varphi:X\rightarrow \mathbb{R}}

\newcommand{\supp}{\operatorname{supp}}

\newcommand{\GenF}{\lim_{T\rightarrow\infty}\frac{1}{T}\int_0^T\varphi(\phi^t(x))dt=\int\varphi{d\mu}}

\def \RR {{\mathbb R}}
\def \SS {{\mathbb S}}

\def \cP {{\mathcal P}}

\def \cU {{\mathcal U}}

\def \fR {{\mathfrak R}}

\def \fM {{\mathfrak M}}

\title{On Bowen's entropy inequality and almost specification for flows}
\author{ Maria Jos\'e Pacifico and Diego Sanhueza}
\thanks{This research has been supported [in part] by CAPES -- Finance Code 001 and by CAPES- and CNPq-grants. MJP was partially supported by FAPERJ}
\date{}

\begin{document}
\maketitle

\begin{abstract}

{We study the Bowen topological entropy of generic  and  irregular points for certain dynamical systems. We define the topological entropy of noncompact sets for flows, analogous to Bowen's definition. We show that this entropy coincides with the Bowen topological entropy of the time-1 map on any set. We also show a Bowen's inequality for flows; namely, that the metric entropy with respect to every invariant measure for a continuous flow is an upper bound for the topological entropy of the set of generic points with respect to the same measure, and the equality is always true if the measure is ergodic. We propose a definition of almost specification property for flows and prove that a continuous flow has the almost specification property if  the time-1 map satisfies this property. Using Bowen's inequality for flows, we show that every continuous flow with the almost specification property is saturated, extending a result of Meson and Vericat in \cite{Meson-Vericat}. Under the same hypotheses, we extend a result of  Thompson on the entropy of irregular points in \cite{Thompson(2012)}.}
\end{abstract}

\section{Introduction}
A central role in ergodic theory is played by Birkhoff averages, a local property which indicates the asymptotic mean of observations along trajectories.
Given  a discrete dynamical system  $(X,f)$, a point $x\in X$ is  {\it generic} with respect to an invariant measure $\mu$ if
$$
\lim_{n\rightarrow\infty}\frac{1}{n}\sum_{j=0}^{n-1}\varphi(f^j(x))=\int\varphi{d\mu}
$$
for all {continuous} observable $\varphi:X\rightarrow\mathbb{R}$. This relation implies that the orbit $\{f^{n}(x):x\in\mathbb{N}\,\,\mbox{ (or}\,\, \mathbb{Z}\mbox{)}\}$ of a 
{$\mu$}-generic point $x$ is uniformly distributed on $X$ {with respect to $\mu$}.
Analogously, a point $x\in X$ is  {\it generic} with respect to an invariant measure $\mu$ for a continuous dynamical system $(X,\Phi=\{\phi^t\}_{t\in\mathbb{R}})$ if
$$
\lim_{T\rightarrow\infty}\frac{1}{T}\int_{0}^{T}\varphi(\phi^t(x)) dt=\int\varphi{d\mu}
$$
for all {continuous} observable $\varphi:X\rightarrow\mathbb{R}$.
Generic points give relevant information about the observable properties of the dynamics, and generic points for different measures {allow} to obtain  complementary data.
It turns out that, from  Birkhoff's Ergodic Theorem, the set of generic points with respect to an ergodic measure $\mu$, denoted $G_\mu(f)$
when $f$ is a discrete dynamic system (and $G_\mu(\Phi)$ for a continuous flow), has {$\mu$-}full measure, and thus reflects the global behavior of the system.
In contrast, if $\mu$ is non-ergodic then generic points form a set of zero measure (can be even empty). However, in several cases, this set can be large from other points of view (for instance, it is dense in the support of $\mu$ when the system has the asymptotic average shadowing property {\cite{Dong-Tian-Yuan(2015)}}). The results in \cite{Colebrook,Eggleston}, which relate the Hausdorff dimension of generic points of an invariant measure and its metric entropy, evidence the importance  { of size} the sets $G_\mu(f)$ and $G_\mu(\Phi)$ from the {dimension} point of view.
{In}  \cite{Bowen(1973)}, Bowen defined the entropy $h(f,Z)$ of a discrete dynamical system $(X,f)$ along with any subset $Z$, {not necessarily compact nor invariant}, of the compact metric space $X$. Topological entropy for noncompact sets has been the focus of many studies since the pioneering work of Bowen.
It  plays a crucial role in many aspects of the ergodic theory and dynamical systems, especially  with the dimension theory and multifractal analysis (see \cite{Barreira2} and references therein). This notion, which resembles the definition of the Hausdorff dimension, is today known as {\textit{Bowen Topological Entropy}}.
{It was proved in \cite[Theorems 2 and 3]{Bowen(1973)}} a remarkable result which says that the metric entropy {with} respect to an invariant measure $\mu$ is an upper bound for the entropy of $G_\mu(f)$; that is, for  any $\mu\in \mathfrak{M}(f)$ (where $\mathfrak{M}(f)$ stands for the set of all the f-invariant measures) it holds
\begin{equation}\label{desBowenint}
h(f,G_\mu(f))\leq h_\mu(f), %\,\, \mbox{for any }\mu\in\mathfrak{M}(f),
\end{equation}
and the equality is always true if $\mu$ is ergodic. 
%{Here $\mathfrak{M}(f)$ is the set of all $f$-invariant measures.}
The inequality  (\ref{desBowenint}) is strict if the measure is not ergodic. Indeed it is not difficult {to} construct examples
of dynamical systems with positive metric entropy with respect to an invariant measure $\mu$ and $G_\mu(f)=\emptyset$.
But, surprisingly, for many important dynamical systems, the equality is valid for all invariant measures. Systems for which the equality  holds  for all invariant measures are called {\it saturated}.
At  \cite{Fan-Liao-Peyriere}, Fan, Liao and Peyri\`ere proved that systems with the specification are saturated.
Pfister and Sullivan at \cite{Pfister-Sullivan(2007)} introduce the notion of $g$-almost product (which is  { satisfied} for all the $\beta-$shifts) and show that they are saturated, generalizing the result of Fan \textit{et al}.
In the sequel  A. Mes\'on and F. Vericat, \cite{Meson-Vericat}, prove that systems with the property of almost specification are also saturated. 
Recall that specification implies the $g$-almost product property, and {this latter} implies almost specification.

\indent In contrast, not much is known about these questions in the setting of continuous flows. We are { interested in proving, on compact metric spaces, the inequality
(\ref{desBowenint}) for continuous flows} and  knowing when a continuous flow is saturated.
To do so, we extend the definition of topological entropy to flows and prove an Abramov-type formula:
\begin{maintheorem}\label{teoA}
Let $\Phi=\{\phi^t\}_{t\in\mathbb{R}}$ be a continuous flow on a compact metric space $X$.~Then
$$
t\cdot h(\Phi,Y)=h(\phi^t,Y)\qquad \mbox{for any }\, Y \subseteq X\mbox{ and }t\geq0.
$$
\end{maintheorem}
Then, we use the above result to prove our main result  which is  { analogous} to Bowen's inequality for flows:

\begin{maintheorem}\label{teoB}Let $\Phi=\{\phi^t\}_{t\in\mathbb{R}}$ be a continuous flow on a compact metric space $X$.~Then
\begin{equation}\label{BowenIneqFlow}
h(\Phi, G_\mu(\Phi))\leq h_\mu(\Phi)
\end{equation}
for { every} $\Phi-$invariant measure $\mu$, and the equality is always true if $\mu$ is ergodic.
\end{maintheorem}

{Theorem \ref{teoB} gives a version for flows of the inequality (\ref{desBowenint})
and  generalizes the main result in \cite{Wang-Chen-Lin-Wu}.}
Our approach is {based on} comparing the generic points for the flow with those of the time-1 map (see Theorem \ref{teoC} and \ref{conten}).
Note that this task is not evident { because an ergodic measure for the flow may be not ergodic for the time-1 map.}
This implies that generic points for the flow with respect to $\mu$ may not be  a generic points for the time-1 map with respect to $\mu$.
Theorem \ref{teoB}, gives a sufficient condition for a flow to be saturated in terms of its time-one map.

\begin{maintheorem}\label{teoC} Let $\Phi=\{\phi^t\}_{t\in\mathbb{R}}$ be a continuous flow on a compact metric space $X$. If $\phi^t$ is saturated, for some $t> 0$, then $\Phi$ is saturated.
\end{maintheorem}
The reciprocal of Theorem \ref{teoC} does not hold, as  shown in Remark \ref{r-teoCnvale}(b).

To present examples of saturated flows, we derive, into our context, some well known results for discrete dynamical systems satisfying a specification type property.
We propose a definition of almost specification for flows and prove Theorem \ref{teoD}, establishing that a flow has the almost specification property if and only if its time-$t$ map does.

\begin{maintheorem}\label{teoD}
A continuous flow $\Phi$ satisfies the almost specification property if and only if its time-$t$ map $\phi^t$ satisfies the almost specification property for any $t\neq0$.
\end{maintheorem}

Given a dynamical system $f:X\to X$, it is also interesting to understand the behavior of the complement $\displaystyle{\hat{X}(f)=\bigcup_{\varphi\in C(X)}\hat{X}(f,\varphi)}$ of $G_\mu(f)$ on $X$.
A point $x\in \hat{X}(f,\varphi)$ if the Birkhoff average $\frac{1}{n}\sum_{j=0}^{n-1}\varphi(f^j(x))$ does not converge and,
in this case, we say that $x$ is   {\it irregular} (or ``{\it non-typical}'') with respect to $\varphi:X\rightarrow\mathbb{R}$.
Similar definitions are given for flows. Again, by  Birkhoff's Ergodic Theorem, the set $\hat{X}(f,\varphi)$ is not detectable by invariant measures. Nevertheless, it can be large from the dimension point of view. In \cite{Barreira-Schmeling(2000)}, the authors proved that, for certain dynamical systems (e.g. subshifts of finite type), $\hat{X}(f,\varphi)$ carries full topological entropy {when it is nonempty}; that is, $h(f,\hat{X}(f,\varphi))=h(f)$. { More generally, the same is proved in \cite{Thompson(2012)} for continuous { maps with} the almost specification property (see also \cite{Ercai-Kupper-Lin,Thompson(2010)}).}
{ As a consequence of Theorem \ref{teoD}, we extend the results of Meson and Vericat \cite{Meson-Vericat} and Thompson  \cite{Thompson(2012)} {for flows}. In particular, we obtain that topologically mixing Anosov flows are saturated, and the set of irregular points carries total entropy. This class of examples includes all the geodesic flows on closed manifolds with negative curvature.
}

This paper is organized as follows: Section 2 summarizes the basic properties of the set of generic points for continuous maps and flows. In Section 3, we recall the definition of topological and metric entropy, we define entropy for flows on noncompact sets and prove the Theorem \ref{teoA}. In Section 4, we give the proof of Theorem \ref{teoB}. Finally, in Section 5, we derive
Theorems \ref{teoC} and \ref{teoD} and extend the results in \cite{Pfister-Sullivan(2007)}
and in \cite{Thompson(2012)} {for flows}. 
\section{Generic points}

%PAREI AQUI

This section collects some relevant properties of the set of generic points with respect to an invariant measure.
We establish a relationship between generic points for a flow and for { its} time-1 map.
{Some} results in this section are probably folklore, but to the best of our knowledge, neither a proof {n}or a {precise} statement {has ever appeared} in the literature, and we include their proofs for completeness.

Throughout this paper,  $(X,d)$ is a compact metric space, $f:X\rightarrow X$ a continuous map and $\Phi:\mathbb{R}\times {X}\rightarrow X$  a continuous flow on $X$.
 We write  $S$ to refer to $f$ or $\Phi$ when no confusion can arise. 
This will help us avoid repeating definitions in these different settings. 
We denote the nonempty, compact and convex set of all the $S$-invariant Borel probability measures endowed with the weak star topology by $\mathfrak{M}(S)$.

Recall that a measure $\mu$ is ergodic if $\mu(A)\in\{0,1\}$ whenever $A$ is {an $S$-}invariant {Borel set}.

\begin{defi} \textnormal{(i)} A point $x\in X$ is a {\bf generic point} for $f$ {with} respect to $\mu\in\mathfrak{M}(f)$ if for { every} continuous function  $\OB$ it holds
$$
\Gen.
$$
\indent\textnormal{(ii)} A point $x\in X$ is a {\bf generic point} {for} $\Phi$ {with} respect to $\mu\in\MF$ if
$$
\GenF \quad \mbox{for { every} continuous function}\quad \OB.
$$
\end{defi}

{The set of generic points for $S$ ($S=f\mbox{ or }\Phi$) with respect to $\mu$ is denoted by $G_\mu(S)$. We say that a point $x$ is {\bf quasi-regular} for $S$ if it is generic for some invariant measure $\mu\in\MS$ and denotes by $Q(S)\mathrel{\mathop:}=\displaystyle{\bigcup_{\mu\in\mathfrak{M}(S)} G_\mu(S)}$ the set of all the quasi-regular points.}

{The following properties {can be found} in \cite[Chapter 4]{Denker-Grillenberger-Sigmund}:}

\begin{enumerate}\label{propriedadegenerico}
\item[(a)] $G_\mu(S)\cap G_\nu(S)=\emptyset$ if $\mu\neq\nu$.\\
\item[(b)] The set $G_\mu(S)$ is a  Borel set. Moreover, $\mu(G_\mu(S))=1$ if $\mu$ is ergodic and  $\mu(G_\mu(S))=0$ otherwise. {In particular, $Q(S)$ has full measure with respect to any {$S$-invariant measure}. Note also that $G_\mu(S)$ can be empty and if $S$ is uniquely ergodic with {a} unique invariant measure $\mu$, { then} $G_\mu(S)=X$.}
\end{enumerate}

We recall that $\mu$ is invariant by the flow  $\Phi$ if $\mu$ is invariant by $\phi^t$ for all $t\in\mathbb{R}$.

\begin{ex} Let $\mathbb{S}^1=\mathbb{R}/\mathbb{Z}$ be the unit circle and  $\Phi:\mathbb{R}\times\mathbb{S}^1\rightarrow\mathbb{S}^1$ be given by
$$
\Phi(t,x)=R_t(x),
$$
where $R_t(x)$ is the rotation {with}  angle $t:[t]+x$. 
This flow is uniquely ergodic with { a} unique invariant measure, the Lebesgue measure $\lambda$.
Then  $G_\lambda(\Phi)=\mathbb{S}^1$ but, for $t\in \Q, \,\,G_\lambda(\phi^t)=\emptyset$ and for $t\in \RR\setminus \Q,\,\,G_\lambda(\phi^t)=\mathbb{S}^1$.
\end{ex}
This shows that $\mu$-generic points for the flow $\Phi$ are not necessarily { $\mu$-generic} for $\phi^t$.
But the proposition below gives a relation between these sets. We denote $C(X):=\{\varphi:X\to \mathbb{R} :\varphi\mbox{ is continuous}\}$ endowed with the uniform topology.

\begin{prop}\label{conten}
If $\mu\in\mathfrak{M}(\Phi)$ then $G_\mu(\phi^t)\subseteq G_\mu(\Phi)$ for all 
$t\in\mathbb{R}\setminus \{0\}$.
\end{prop}
\begin{proof} It is enough verify {for}
 $t=1$. {Let  $x\in G_\mu(\phi^1)$  and $\varphi\in C(X)$. Then $\phi^j(x) \in G_\mu(\phi^1)$ for all  $ j \geq 1.$ If $[T] $ denotes the integer part of $T\in \RR$, we have }

 %\textcolor{red}{Then the orbit of $x$, $\phi^j(x), \,\,j \geq 1, $ also belongs to $G_\mu(\phi^1)$ and it holds}
 %\textcolor{red}{Then $\phi^j(x) \in G_\mu(\phi^1), \,\,\forall \,\, j \in \NN}$. AQUI
 %Then it holds
$$
\lim_{T\rightarrow\infty}\frac{1}{T}\int_0^T\varphi(\phi^t(x))dt  =
\lim_{T\rightarrow\infty}\frac{1}{T} \left(\sum_{j=0}^{[T ] -1}\int_0^1(\varphi\circ\phi^t)(\phi^jx)dt + \int_{[T]}^T \varphi(\phi^t(x))dt \right) =
$$
$$
=\int_0^1\lim_{T\rightarrow\infty}\frac{1}{T}\sum_{j=0}^{[T]-1}(\varphi\circ\phi^t)(\phi^jx)dt=
\int_0^1\int(\varphi\circ\phi^t)d\mu{dt}=
 \int\varphi{d\mu}.
 $$
This ends the proof.
\end{proof}
\begin{remark}\label{Gnonempty}
If $\mu \in \MF$ and $G_{\mu}(\phi^{1})\neq \emptyset$ then $G_\mu(\Phi)\neq \emptyset$.
\end{remark}
%%%%%%%%%%%%%%%%%%%%%%%%%%%%%%%%%%%%%%%%%%%
%%%%%%%%%%%%%%%%%%%%%%%%%%%%%%%%%%%%%%%%%%%%

Invariant measures for the time-$t$ map are not necessarily invariant for the flow. {But, any $\mu\in\Mft$ produces a measure $\overline{\mu}\in\MF$ defined by  $\overline{\mu}:=\int_0^1\phi^s_*\mu{ds}$, where
$\phi^s_*\mu$ is {given}  by $\phi^s_*\mu(B):=\mu(\phi^{-s}(B))$ for all Borel sets $B$.}
The construction of  $\overline{\mu}$  is a standard one in the ergodic theory of flows, see \cite[p. 968]{Walters8}.

{We recall that $Q(S)$ is the set of quasi-regular points for $S$, where $S$ stands for $f$ or $\Phi$.}
The following result establishes that $Q(\Phi)$ coincides with $Q(\phi^{1})$.

\begin{teo}\label{G=g}
Let $\Phi$ be a continuous flow defined on a compact metric space $X$. Then $Q(\Phi)=Q(\phi^t)$ for all $t\in\mathbb{R}\setminus\{0\}$.
\end{teo}
\begin{proof} {It is enough to prove for $t=1$.}
We prove first that {$Q(\Phi)\subseteq Q(\phi^1)$.}
Let $x\in Q(\Phi)$ and fix a continuous map $\overline{\varphi}:X\rightarrow \mathbb{R}$.
Consider the sequence
\begin{equation}\label{SumBir}
\frac{1}{n}\sum_{j=0}^{n-1}(\overline{\varphi}\circ f^j)(x)
\end{equation}
where $f:=\phi^1$.
{Consider any two increase sequences of positive integers
 $\{n_k^1\}_{k\geq0}$ and  $\{n_k^2\}_{k\geq0}$ such that the limits below exist}
$$
\lim_{k\rightarrow\infty}\frac{1}{n_k^1}\sum_{j=0}^{n_k^1-1}(\overline{\varphi}\circ f^j)(x) \qquad\mbox{and}\qquad\lim_{k\rightarrow\infty}\frac{1}{n_k^2}\sum_{j=0}^{n_k^2-1}(\overline{\varphi}\circ f^j)(x).
$$

As  $C(X)$ is a separable space, we can find subsequences  {of 
$\{n_k^1\}_{k\geq0}$ and  $\{n_k^2\}_{k\geq0}$}
denoted by  $\{n_k^1(x)\}_{k\geq1}$ and  $\{n_k^2(x)\}_{k\geq1}$,
respectively, such that the limits
$$
C_1(\varphi):=\lim_{k\rightarrow\infty}\frac{1}{n_k^1(x)}\sum_{j=0}^{n_k^1(x)-1}(\varphi\circ f^j)(x) \qquad\mbox{and}\qquad C_2(\varphi):=\lim_{k\rightarrow\infty}\frac{1}{n_k^2(x)}\sum_{j=0}^{n_k^2(x)-1}(\varphi\circ f^j)(x)
$$
do exist for all   $\varphi\in C(X)$.

{For $i=1,2$, we also denote by $C_i$ the induced linear operator $C_i:C(X)\rightarrow \mathbb{R}$, $\varphi\mapsto C_i(\varphi)$}.
These operators are continuous, positive defined and satisfy $C_1({\bf 1}_X)=C_2({\bf 1}_X)={\bf 1}_X$ and by the Theorem of Riesz define two measures $\mu_i$ invariant by the time-$1$ map such that
$C_1(\varphi)=\int{\varphi}d\mu_1$ and $C_2(\varphi)=\int{\varphi}d\mu_2$ for all $\varphi\in C(X)$.

Since  $x$ is a quasi-regular point {for $\Phi$,} we have
$$
\lim_{T\rightarrow\infty}\frac{1}{T}\int_0^T\varphi(\phi^t(x))dt
= \lim_{k\rightarrow\infty}\frac{1}{n_k^1(x)}\sum_{j=0}^{n_k^1(x)-1}\int_0^1(\varphi\circ\phi^t)(\phi^jx)dt
= \int_0^1C_1(\varphi\circ\phi^t)dt.
$$
Similarly, $\lim_{T\rightarrow\infty}\frac{1}{T}\int_0^T\varphi(\phi^t(x))dt=\int_0^1C_2(\varphi\circ\phi^t)dt$ and so
$$
\int_0^1C_1(\varphi\circ \phi^t)dt=\int_0^1C_2(\varphi\circ \phi^t)dt.
$$
Since  $C_1$ and $C_2$ are continuous, there is  $\overline{t}=\overline{t}(\varphi) \in [0,1]$ with  $C_1(\varphi\circ \phi^{\overline{t}})=C_2(\varphi\circ \phi^{\overline{t}})$ and we can take $\overline{t}(\varphi)$ depending continuously on $\varphi$.
Define $L:[0,1]\rightarrow[0,1]$ by $L(s)=\overline{t}(\overline{\varphi}\circ{\phi}^{-s})$ and note that $L$ is also continuous.
Then $L$ admits a fixpoint $\overline{s}$ in $[0,1]$. 
Hence, { for the continuous function $\overline{\varphi}\circ \phi^{-\overline{s}}$}, we have
$$
C_1(\overline{\varphi})=C_1(\overline{\varphi}\circ \phi^{-\overline{s}}\circ \phi^{\overline{t}(\overline{\varphi}\circ \phi^{-\overline{s}})})=C_2(\overline{\varphi}\circ \phi^{-\overline{s}}\circ \phi^{\overline{t}(\overline{\varphi}\circ \phi^{-\overline{s}})})=C_2(\overline{\varphi}).
$$
This proves that the sequence (\ref{SumBir}) converges, and this implies  that the  limit
$$
\lim_{k\rightarrow\infty}\frac{1}{n}\sum_{j=0}^{n}(\varphi\circ f^j)(x)
$$
exists for all continuous maps $\varphi:X\rightarrow\mathbb{R}$, and thus we conclude that $x\in Q(\phi^1).$

Next we prove $Q(\phi^1)\subseteq Q(\Phi)$.
Let $x\in Q(\phi^1)$, then $x$ is a generic point to $\phi^1$ with respect to some $\phi^1$-invariant measure $\mu_x$.
Consider the $\Phi$-invariant measure $\overline{\mu}_x:=\int_0^1\phi^t_* \mu_x dt$.

We claim that $x\in G_{\overline{\mu}_x}(\Phi)$. Indeed, for all $\varphi\in C(X)$, it holds
$$
\lim_{T\rightarrow\infty}\frac{1}{T}\int_0^T\varphi(\phi^tx)dt =
\lim_{T\rightarrow\infty}\frac{1}{T} \left(\sum_{j=0}^{[T ]-1}\int_0^1(\varphi\circ\phi^t)(\phi^jx)dt + \int_{[T]}^T \varphi(\phi^t(x))dt \right) 
%\int_0^1\lim_{T\rightarrow\infty}\frac{1}{T}\sum_{j=0}^{T-1}(\varphi\circ\phi^t)(\phi^ix)dt
= \int\varphi d\overline{\mu}_x.
$$
This shows that $x\in G_{\overline{\mu}_x}(\Phi)\subseteq Q(\Phi)$ as claimed.

{All together, complete the proof of the Theorem \ref{G=g}}.
\end{proof}

\begin{remark}\label{r.medidamubarra}
For each $\mu \in \fM(\phi^{1})$ the set $G_{\mu}(\phi^{1})$ is entirely contained in the set $G_{\overline{\mu}}(\Phi)$, where $\overline{\mu}:=\int_{0}^{1} \phi^{t}_{*}\mu dt.$
\end{remark}

A useful characterization of generic points is obtained from the {\it empirical measures}.
{Recall that given $x\in X$, }
to each $n\in\mathbb{N}$, $\xi_n(x)=\frac{1}{n}\sum_{j=0}^{n-1}\delta_{f^j(x)}$ is a Borel  probability called {\bf empirical measure} concentrated in the  orbit of $x$.
We denote by $V_f(x)$ the set of all limit points of the sequence  $\{\xi_n(x)\}_{n\geq1}$; that is,
$$
V_f(x)=\{\mu:\mbox{ exists }\,\, \{n_k\}_{k\geq1}\mbox{ such that }\xi_{n_k}(x)\rightarrow\mu\mbox{ as }k\rightarrow\infty\},
$$
where the convergence is in the weak* topology in $\mathfrak{M}(X)$.
The set  $V_f(x)$ is non empty, compact, connected, and  $V_f(x)\subset \mathfrak{M}(f)$ {(see \cite{Denker-Grillenberger-Sigmund}). Note that $x \in G_{\mu}(f)$ if and only if $V_{f}(x)=\{\mu\}$. We shall use this characterization of generic points later in the proof of the Theorem \ref{teoB}.}

\section{Entropy for non compact sets and proof of Theorem \ref{teoA}}
Furstenberg proves in  \cite{Furstenberg2} that if  $f:\SS^1\rightarrow\SS^1$ is defined by  $f(z)=z^n$ and  $Y\subseteq\SS^1$
 is a compact positive invariant set, then the Hausdorff dimension of $Y,\,\, \dim_H(Y)$ satisfies
$
\dim_H(Y)=h(f|_Y)/\log{n},
$
and if  $\mu\in\mathfrak{M}(f)$ is ergodic, Colebrook proves in  \cite{Colebrook} that {$
\dim_H(G_\mu(f))=h_\mu(f)/\log(n).
$}

Motivated by these results, given a continuous map  $f:X\rightarrow X$, where $X$ is a compact metric space,  Bowen introduced in \cite{Bowen(1973)}  the notion of entropy for subsets of $X$, {neither necessarily compact nor invariant.}
In this section, we propose a definition of entropy for non compact sets in the setting of flows and prove
an {\em{Abramov-type}} formula that relates the entropy of a flow over non compact sets with the entropy of the time-t map over this set.
To do so, we proceed as follows.

We start reviewing the concept of metric and topological entropy for a dynamical system and recall some well known results in this case.
{In the sequel, we extend in a natural way {Bowen's} definition of entropy for noncompact sets for flows and { prove} Theorem A.}

\subsection{Entropy of a dynamical system}
In the following  $f : X \to X$ will be a continuous map defined on a {{compact} metric space $X$}.
For $\varepsilon > 0$ and $n \ge 1$,
we consider the {\it dynamical ball} of radius $\varepsilon > 0$ and length $n$ around $x \in X$:
$$
B_n(x, \varepsilon) = \{y \in X: d(f^j(x), f^j(y)) < \varepsilon \mbox{ for every } 0 \le j \le n-1\}.
$$
{Let $r_n(\varepsilon)$ denote the smallest number of dynamical balls of radius $\varepsilon$ and length $n$, covering $X$.
Then the topological entropy $h(f)$ can be defined by
$$
h(f):=\lim_{\varepsilon\rightarrow0}\limsup_{n\rightarrow\infty}\frac{1}{n}\log{r_n(\varepsilon)}.
$$
(see \cite[Chapter 7]{Walters}).}

{Let $\mu$ be an invariant measure and { let $\cP$ be} a finite, measurable partition
of $X$. The {\it metric entropy} of $\mu$ corresponding to the partition $\cP$ is defined as
$$
h_\mu(\cP) := -\lim_{n\to\infty}\frac{1}{n} \sum_{B\in\cP^{n}}\mu(B) \log \mu(B),
$$
where $\cP^1=\cP$ and $\cP^{n}$ is the $n$th joint of $\cP$:
$
\cP^{n} = \cP \vee f^{-1}\cP \vee \cdots \vee f^{-(n-1)}\cP.
$
The {\bf metric entropy} of $f$ with respect to $\mu$ is defined as
$$
h_{\mu}(f) := \sup\{h_{\mu}(\cP): \cP \mbox{ is a finite partition of } X\}.
$$
}
{
\indent Recall that two systems $(X_1,\mathcal{B}_1,\mu_1,f_1)$ and $(X_1,\mathcal{B}_1,\mu_1,f_2)$ are  {\bf isomorphic} if there is $M_i\subseteq X_i$, with $\mu_i(M_i)=1$ and $f_i(M_i)\subseteq M_i$, $i=1,2$, and there exists an invertible transformation $h:M_1\rightarrow M_2$ which preserves the measures $\mu_1,\mu_2$ satisfying
$$
(h\circ f_1)(x)=(f_2\circ h)(x)\mbox{ for all }x\in M_1.
$$
The next lemma is a consequence of \cite[Theorem 4.11]{Walters}:
\begin{lemma}\label{l-aux} For each $t\in[0,1]$, the systems $(X,\mathcal{B}(X),\mu,\phi^1)$ and $(X,\mathcal{B}(X),\phi^t_*\mu,\phi^1)$ are isomorphic. In particular, $h_\mu(\phi^1)=h_{\phi^t_*\mu}(\phi^1)$.
\end{lemma}

{
\begin{lemma}\label{LemmaEntropy} $h_{\overline{\mu}}(\phi^1)= h_\mu(\phi^1)$, where 
$\overline{\mu}=\int_0^1\phi^t_*\mu\,dt$.
\end{lemma}

\begin{proof}
{Let $\mathfrak{M}_{\mathrm{erg}}(\phi^1)$ be the set of  invariant ergodic measures of $\phi^1$.}
If $\mu\in{\mathfrak{M}_{\mathrm{erg}}(\phi^1)}$, note that $\phi^t_*\mu\in{\mathfrak{M}_{\mathrm{erg}}(\phi^1)}$, so $\{\phi^t_*\mu\}_{ t}$ is the ergodic decomposition of the measure $\overline{\mu}$. Observe that $h_{\phi^t_*\mu}(\phi^1)=h_\mu(\phi^1)$ (Lemma \ref{l-aux}). By Jacobs's Theorem (\cite[Theorem 8.4]{Walters}), {it} follows that
$$
	h_{\overline{\mu}}(\phi^1)
	=\int_0^1h_{\phi^t_*\mu}(\phi^1)\,dt
	=h_\mu(\phi^1).
$$
To consider the general case, {take}  $\mu\in\Mfone$ {and let } $\mu=\int\mu_{P}\,d\hat\mu (P)$ be its ergodic decomposition (relative to $\phi^1$). Observe that
$$
	\phi^t_*\mu
	=\phi^t_*\int\mu_{P}\,d\hat\mu(P)
	=\int\phi^t_*\mu_{P}\,d\hat\mu(P).
$$
Because  $\mu_P$ is $\phi^1$-ergodic for $\hat\mu$-almost every $P$, $\phi^t_*\mu_{P}$ is also $\phi^1$-ergodic for any $t$ and since
\[
	\bar\mu
	= \int_0^1\phi^t_*\mu\,dt
	= \int_0^1\phi^t_*\int\mu_P\,d\hat\mu(P)\,dt
	= \int_0^1\int\phi^t_*\mu_P\,d\hat\mu(P)\,dt
\]
{it } follows that $\{\phi^t_*\mu_P\}_{t,P}$ is the ergodic decomposition for $\bar\mu$.
Applying Jacobs's Theorem again, we have
$$
	h_{\overline{\mu}}(\phi^1)
	=\int_0^1\int{h_{\phi^t_*\mu_{P}}(\phi^1)}\,d\hat\mu(P)\,{dt}
	=\int_0^1\int{h_{\mu_{P}}(\phi^1)}\,d\hat\mu(P)\,{dt}
	= \int_0^1h_{\mu}(\phi^1)\,dt
	=h_{\mu}(\phi^1).
$$
 The statement of the lemma follows.
\end{proof}
}

\begin{defi} Let $\Phi$ be a continuous flow defined on a compact metric space $(X,d)$. The {\bf{topological entropy}} (resp. {\bf{metric entropy}}) of $\Phi$, $h(\Phi)$ (resp. $h_\mu(\Phi)$),
is the topological entropy (resp. metric entropy) of its time-one map $\phi^{1}$.
\end{defi}
\noindent In connection with topological entropy, the variational principle establishes that
\begin{center}
$
h(S) = \sup\{h_{\mu}(S):\mu\in \mathfrak{M}(S)\}.
$
\end{center}

\begin{defi} Let $(X,S)$ be a dynamical system. We say that $\mathfrak{M}_{\mathrm{erg}}(S)$ is \textbf{entropy dense} in $\mathfrak{M}(S)$, if for every $S$-invariant $\mu$ there is a sequence of ergodic $S$-invariant measures $\{\mu_n\}_{n\geq1}$ converging to $\mu$ such that  $h_{\mu_n}(S)$ converges to $h_\mu(S)$.
\end{defi}

\begin{teo}\label{Dense-Entropy} Let $\Phi$ be a continuous flow defined on a compact metric space $X$. If $\mathfrak{M}_{\mathrm{erg}}(\phi^1)$ is entropy dense, then $\mathfrak{M}_{\mathrm{erg}}(\Phi)$ is entropy dense.
\end{teo}
\begin{proof}  Let  $\mu\in\mathfrak{M}(\Phi)$. In particular,  $\mu\in\mathfrak{M}(\phi^1)$ and since $\mathfrak{M}_{\mathrm{erg}}(\phi^1)$ is entropy dense in $\mathfrak{M}(\phi^1)$, there is a sequence $\mu_n\in\mathfrak{M}_{\mathrm{erg}}(\phi^1)$ converging to $\mu$ (in the weak$\empty^*$ topology) and  $h_{\mu_n}(\phi^1)\rightarrow h_\mu(\phi^1)$ as $n\rightarrow\infty$.  {For each $n \geq 1,$ consider the measure}
$$
\overline{\mu}_n:=\int^1_0\phi^{t}_*\mu_ndt.
$$
Since  $\mu_n$ is ergodic respect to $\phi^1$, the measure $\overline{\mu}_n$ is ergodic respect to  $\Phi$, for all $n\geq1$.

\noindent {\bf Claim.} \/ $\overline{\mu}_n$ converges to $\mu$.

\noindent \textit{Proof of the Claim.}
Let $\varphi\in C(X)$. The function  $\displaystyle{t\mapsto\int\varphi{d\phi^{t}_*\mu_n}}$ is Lebesgue measurable, and applying the Dominated Convergence Theorem, we get

\begin{eqnarray*}
\lim_{n\rightarrow\infty}\int\varphi{d\overline{\mu}_n} & = & \lim_{n\rightarrow\infty}\int_0^1\int\varphi{d\overline{\mu}_n}dt
 =  \int_0^1\lim_{n\rightarrow\infty}\int\varphi{d\overline{\mu}_n}dt\\
& = & \int_0^1\int\varphi{d\mu}dt
 =  \int\varphi{d\mu}.\qquad\square
\end{eqnarray*}

\noindent Now,  Lemma \ref{LemmaEntropy}  implies
$$
\lim_{n\rightarrow\infty}h_{\overline{\mu}_n}(\Phi) = \lim_{n\rightarrow\infty}h_{\overline{\mu}_n}(\phi^1) = \lim_{n\rightarrow\infty}h_{\mu_n}(\phi^1) = h_\mu(\phi^1) = h_\mu(\Phi).
$$
This shows that $\mathfrak{M}_{\mathrm{erg}}(\Phi)$ is entropy dense in $\MF$.%\qquad$
\end{proof}

It is clear that if $\mathfrak{M}_{\mathrm{erg}}(S)$ is entropy dense, then  $\mathfrak{M}_{\mathrm{erg}}(S)$ is dense in $\mathfrak{M}(S)$ (but the converse is not true as shown in \cite[Proposition 4.29]{Gelfert-Kwietniak(2018)}).

\subsection{Topological entropy for non compact sets for flows and proof of Theorem \ref{teoA}}
This section proposes a definition of topological entropy for non compact sets for continuous flows and proves  Theorem \ref{teoA}.
To this end,
let $\Phi$ be a continuous flow on a compact metric space   $X$ and   $\cU$ be a finite open cover of $X$.
If  $B\subseteq X$, we set $B\prec\cU$ whenever $B$ is contained in some subset of  $\cU$. Analogously, for any collection $\{B_i\}_{i\in I}, B_i\subset X$, we set $\{B_i\}_{i\in I}\prec \cU$ if each $B_i$ is contained in some element of $\cU$.

Define the quantities:
\begin{enumerate}
\item[(a)]$N(B)=N(\Phi,\cU,B)$ is the biggest non negative number such that
\ $\phi^t(B)\prec\cU$ for all $0\leq t<N(B)$. We put
$N(B)=0$ if $B\nprec\cU$ and $N(B)=\infty$ if $\phi^t(B)\prec\cU$ for all $t\geq0$.
\item[(b)]$ D(B):=D(\Phi,\cU,B)=\left\{
\begin{array}{cl}
\exp(-N(B)) & \mbox{ if } N(B) \mbox{ is finite}\\
0 &  \mbox{ if }  N(B)=\infty
\end{array}\right.
$
\item[(c)] {$S_\alpha(\mathcal{B}):=S_\alpha(\Phi,\cU,\mathcal{B})=\sum_{i\geq1}D(B_i)^\alpha$ if $\alpha\in\mathbb{R}$ and $\mathcal{B}=\{B_i\}_{i\geq1}$ is a countable collection of subsets of ~$X$.}
\end{enumerate}
For $Y\subseteq X$, define
{$
\Gamma_\alpha(Y):=\Gamma_\alpha(\Phi,\cU,Y)=\lim_{\varepsilon\rightarrow0}\inf\{S_\alpha(\mathcal{E}):\mathcal{E}=\{E_i\}_{i\geq1}\mbox{ is a cover of }Y \mbox{ with }D(E_i)<\varepsilon\}.
$}
Note that $\Gamma_\alpha$ is an exterior measure and satisfies
$\Gamma_\alpha(Y)\leq\Gamma_{\overline{\alpha}}(Y)$ if $\alpha>\overline{\alpha}$. Moreover, $\Gamma_\alpha(Y)\not\in\{0,\infty\}$ for at most one unique value  $\alpha$.
Next, define
$$
h_\cU(\Phi,Y):=\inf\{\alpha:\Gamma_\alpha(Y)=0\}.
$$
Finally, the {\bf topological entropy of} $\Phi$ {\bf over} $Y$ is defined by
$$
h(\Phi,Y):=\sup\{h_\cU(\Phi,Y):\cU\mbox{ is a finite open cover of }X\}.
$$
For discrete dynamical systems $(X,f)$  similar quantities are defined, 
and are denoted by  $n(B)=n(f,\cU,B), d(B)=d(f,\cU,B), s_\alpha(\mathcal{B})=s_\alpha(f,\cU,\mathcal{B})$ and $\gamma_\alpha(Y)=\gamma_\alpha(f,\cU,Y)$ respectively. {See {\cite{Bowen(1973)}} for more details.}

{\bf{Proof of Theorem \ref{teoA}.}}\/
Fix $\tau>0$ and let
$\cU=\{U_1,U_2,...,U
_\ell\}$ be a finite open cover of $X$ and $\varepsilon,\alpha>0$.
Next, we prove that if  $Y\subset X$ then  $\tau h(\Phi,Y)\geq h(\phi^\tau,Y)$ and $\tau h(\Phi,Y)\leq h(\phi^\tau,Y)$. For this, we proceed as follows.

Let $\mathcal{B}=\{B_i\}_{i\geq1}$ be a countable cover of $Y$ satisfying  $D(B_i)<\varepsilon$ for all $i\geq1$.
For each $i\geq1$, it is clear that if $N(B_i)=N(\Phi, B_i)=\infty$ then $n(B_i)=n(\phi^\tau,B_i)=\infty$, and if $N(B_i)<\infty$ then  $n(B_i)<\infty$ and,
by definition, we have
$$
N(B_i)\leq \tau n(B_i), \,\, \mbox{which implies }\,\,S_\alpha(\mathcal{B})\geq s_\alpha(\mathcal{B}) \,\,\mbox{and so}\,\,
\Gamma_\alpha(Y)\geq\gamma_{\tau\alpha}(Y).
$$
Thus, taking the infimum over $\alpha$, we get
$
h_{\cU}(\Phi,Y)\geq\frac{1}{\tau}h_{\cU}(\phi^\tau,Y)
$
and taking the supremum over all finite open covers $\cU$ of $X$,
$
h(\Phi,Y)\geq \frac{1}{\tau}h(\phi^\tau,Y).
$\\
{To prove the opposing inequality,} let  $\tilde{\delta}=\tilde{\delta}(\cU)$ be the Lebesgue number of  $\cU$ and $\delta=\delta(\tau,\tilde{\delta})$ such that $d(\phi^t(x),\phi^t(y))<\tilde{\delta}/2$ if $d(x,y)<\delta$ and $0\leq t\leq\tau$.
Let $\V=\{B(x,\delta/2):x\in X\}$. 
Clearly, $\V$ is an open cover of $X$.
 Let ${\cU}'\subseteq\V$ be a {finite sub-cover} of  $X$ and {let  $\C=\{C_i\}_{i\geq1}$  be  }a cover of $Y$ such that {$d(C_i)=d(\phi^\tau,{\cU}',C_i)<\varepsilon$} for all $i\geq1$.

Shrinking $\varepsilon$, if necessary, we can assume that
each  $C_i$ is contained in some ball with  radius  $\delta/2$.
 Then, if  $y,z\in C_i$ we have $d(\phi^t(z),\phi^t(y))<\tilde{\delta}/2$ for all $0\leq t\leq\tau$ and so
$
\phi^t(C_i)\subseteq B(x',\tilde{\delta}/2)\subseteq U_j
$
for some $x'\in X$, some $U_j\in \U$, and all $0\leq t\leq\tau$.
{This implies that
$
N(\Phi,\cU,C_i)\geq \tau n(\phi^\tau,\cU',C_i)\mbox{ for all }i\geq1,
\mbox{and hence } 
D(\Phi,\cU,\C)\leq d(\phi^\tau,{\cU}',\C) \Rightarrow \Gamma_\alpha(\Phi,\cU,Y)\leq \gamma_{\tau\alpha}(\phi^\tau,\cU',Y).
$
{Thus}
$
h_{\cU}(\Phi,Y)\leq \frac{1}{\tau}h_{{\cU}'}(\phi^\tau,Y).
$}
Taking the supremum over all finite open covers  $\cU$ of  $X$, we get
$$
h(\Phi,Y)\leq \frac{1}{\tau}h(\phi^\tau,Y),
$$
finishing the proof of Theorem \ref{teoA}. $\quad\quad \square$

The above theorem shows that our definition of entropy for flows over a subset $Y\subseteq X$ using the approach given by Bowen coincides with the one provided by
the authors in  \cite{Shen-Zhao}, following the approach  of Pesin and Pitskel \cite{Pesin-Pitskel}.
Although both approaches are based on properties of the dimension of Carath\'eodory, {these}  constructions are distinct. One of the main differences is that we use $D(\Phi, \cU, E)$ 
to estimate the entropy. In contrast, Shen and Zhao use a different quantifier given by certain $(T,\epsilon)$-separated sets to define the entropy, see \cite{Shen-Zhao}. We also point out that  Burns and Gelfert in \cite{Burns-Gelfert(2014)} consider the entropy of a flow on non compact sets 
$Z\subseteq X$ as the entropy of its time-one map on $Z$. Although our definition is distinct from that in \cite{Burns-Gelfert(2014)}, Theorem \ref{teoA} implies that these values are the same.  In particular, Theorem \ref{teoA} implies  Abramov's formula used by Burns and Gelfert.  { On the other hand, it was pointed out in \cite[p. 38]{Barreira-Schmeling(2000)} that there are examples such as $h(f^{-1},Y)\neq h(f,Y)$. Thus,  Theorem~ \ref{teoA} does not hold for $t<0$.}

The following properties are straightforward consequences of  Theorem \ref{teoA}:

\begin{prop}\label{PropEntNon} Let  $\Phi$ be a continuous flow, then\\
\indent \textnormal{(i)} $h(\Phi,X)=h(\Phi)=h(\phi^1)$.\\
\indent \textnormal{(ii)}  If $Y\subseteq Z$ then $h(\Phi,Y)\leq h(\Phi,Z)$.\\
\indent \textnormal{(iii)}  If $Y=\cup_{i\geq1}Y_i$ then  $h(\Phi,Y)=\sup_{i\geq1}\{h(\Phi,Y_i)\}$, where $Y_i\subseteq X$.\\
\indent \textnormal{(iv)}  $h(\Phi,\emptyset)=0$.
\end{prop}

%%%%%%%%%%%%%%%%%%%%%%%%%%%%%%%%%%%%%%%%%%%%%%
%%%%%%%%%%%%%%%%%%%%%%%%%%%%%%%%%%%%%%%%%%%%%%
\section{Proof of Theorem \ref{teoB}}
This section proves  Theorem B,
which establishes $h(\Phi,G_\mu(\Phi))\leq h_\mu(\Phi)$
for any $\Phi$-invariant measure.
We start defining the set $\fR(h_\mu(\Phi))$~as
$$
\fR(h_\mu(\Phi)):=\{x\in X: \exists\,\,\nu\in V_{\phi^1}(x)\mbox{ such that }h_\nu(\phi^1)\leq h_\mu(\Phi)\},
$$
where $V_{\phi^1}(x)$ is the set of accumulation points of the empirical measure $\xi_n(x)=\frac{1}{n}\sum_{j=0}^{n-1}\delta_{\phi^j(x)}$.

 Theorem \ref{teoA}  together with
 \cite[Theorem 2]{Bowen(1973)} give 
$$
h(\Phi,\fR(h_\mu(\Phi)))=h(\phi^1,\fR(h_\mu(\phi^1)))\leq h_\mu(\phi^1)=h_\mu(\Phi).
$$

\noindent Then by Proposition \ref{PropEntNon}(ii), we are left to  prove that  $G_\mu(\Phi)\subseteq \fR(h_\mu(\Phi))$.
By {Theorem} \ref{G=g}, each $x\in G_\mu(\Phi)$ induces a Borel probability $\mu_x$ invariant by the time-one map of the flow and such that $x\in G_{\mu_x}(\phi^1)$.
Denote $\mu^t_x:=\phi^t_*\mu_x$ and define $\overline{\mu}_x:=\int_0^1\mu^t_xdt$.

\noindent {\bf Claim.} $\mu=\overline{\mu}_x$. Indeed, for each  $\varphi\in C(X)$, we have that
\begin{eqnarray*}
\int\varphi d\mu  & = &  \lim_{T\rightarrow\infty}\frac{1}{T}\int_0^T(\varphi\circ\phi^t)(x)dt
= \lim_{T\rightarrow\infty}\frac{1}{T} \left(\sum_{j=0}^{[T]-1}\int_0^1(\varphi\circ\phi^t)(\phi^jx)dt + \int_{[T ]}^T \varphi(\phi^t(x))dt \right)=\\
&=& \int_0^1\lim_{T\rightarrow\infty}\sum_{j=0}^{[T]-1}(\varphi\circ\phi^t)(\phi^jx)dt
 = \int_0^1\int(\varphi\circ\phi^t)d\mu_xdt
 =  \int_0^1\int\varphi d\mu_x^tdt
=  \int\varphi d\overline{\mu}_x. \quad\quad \square
\end{eqnarray*}

Lema \ref{LemmaEntropy} {, together with the claim, implies} $h_{\mu_x}(\phi^1)= h_{\overline{\mu}_x}(\phi^1)=h_\mu(\phi^1)=h_\mu(\Phi)$ and so $x\in \fR(h_\mu(\Phi))$ by definition  of $\fR(h_\mu(\Phi))$, proving that $h(\Phi,G_\mu(\Phi))\leq h_\mu(\Phi)$.

{If $\mu\in \EF$, then $\mu(G_\mu(\Phi))=1$ and using \cite[Theorem 1]{Bowen(1973)}, we get that  $h_\mu(\phi^1)\leq h(\phi^1,G_\mu(\Phi))$, which implies, by Theorem \ref{teoA},  that $h_\mu(\Phi)\leq h(\Phi,G_\mu(\Phi))$ and the equality follows. The proof of Theorem \ref{teoB} is complete.}
$\square$

\begin{remark} If $\mu\in\mathfrak{M}_{\mathrm{erg}}(\Phi)$ and $t> 0$, then the set $G_\mu(\Phi)$ satisfies
$$
{\frac{1}{t}h_\mu(\phi^t)}=h_\mu(\phi^1)=h(\phi^1,G_\mu(\Phi)),
$$
even if $\mu$ is not an ergodic measure for $\phi^t$. This {means}  that we can exhibit a set whose topological entropy allows us to determine the metric entropy of the map $\phi^t$.
\end{remark}

\section{Saturated systems and proof of theorems \ref{teoC} and \ref{teoD}}
%%%%%%%%%%%%%%%%%%%%%%%%%%%%%%%%%%%%
%%%%%%%%%%%%%%%%%%%%%%%%%%%%%%%%%%%%%%

%The {purpose}  of 
This section  provides non-trivial examples of continuous dynamic systems so that {the inequality} (\ref{BowenIneqFlow}) {in Theorem \ref{teoB}}
{ is always}  an equality: the so-called \textit{saturated systems}. 
We need to impose some additional conditions on the system to do so.  {Hence},  we  define  the almost specification property for flows and prove a result establishing that a continuous flow defined {on a compact metric space} satisfies this property if and only if its {time-$t$ map  satisfies} the almost specification property {in the discrete sense}.
This result {allows} us  to obtain that if a flow $\Phi$ has the {\bf almost specification property}, then $\Phi$ is saturated (Theorem \ref{Quase-Spec}(b)), extending 
\cite{Meson-Vericat} to the context of flows.
In addition, this result also makes {it} possible to describe some non-trivial examples of saturated flows.

\begin{defi} A dynamical system $(X,S)$ is {\bf saturated} if
$$
h(S,G_\mu(S))=h_\mu(S)\qquad\mbox{ for each }\mu\in\mathfrak{M}(S).
$$
\end{defi}

All dynamical systems with null entropy are saturated. We say that $(X,S)$ is \textbf{strongly saturated} if it is saturated, and in addition, $G_\mu(S)$ is nonempty for each {$\mu\in\mathfrak{M}(S)$}. It is clear that whenever $X$ is not a {singleton}, then the identity map is saturated but not strongly saturated. On the other hand, uniquely ergodic systems are always strongly saturated. 
It is not difficult {to} construct non trivial examples of systems that are not saturated, and below, we give such {an} example.

\begin{ex}
Let $(X_1,f_1,\mu_1)$ and $(X_2,f_2,\mu_2)$ {be} two uniquely ergodic systems with entropy $h_1,\, h_2>0$, respectively (see for instance  \cite{Hahn-Katznelson(1967)}). Consider the dynamical system  $(X,f)$ defined by $X=X_1\sqcup X_2$  and  $f(x)=f_i(x)$ if $x\in X_i,\, i= 1, 2$. For $\mu=\frac{1}{2}(\mu_1+\mu_2)$, $G_\mu(f)=\emptyset$ but $h_\mu(f)=\frac{1}{2}(h_1+h_2)>0=h(f,G_\mu(f))$.
\end{ex}

Next, we prove Theorem \ref{teoC}, which gives a sufficient condition to determine if a continuous flow is saturated. This establishes that if the time-$t$ map of a continuous flow is saturated for some $t>0$, so is the flow.

\noindent {\bf{Proof of Theorem \ref{teoC}.}}\/ Let $\mu\in\MF$. By Proposition \ref{conten} and Theorem \ref{teoB} we get
$$
h_\mu(\phi^t)=h(\phi^t,G_\mu(\phi^t))=t\cdot h(\Phi,G_\mu(\phi^t))\leq t\cdot h(\Phi,G_\mu(\Phi))\leq t\cdot h_\mu(\Phi)=h_{\mu}(\phi^t).
$$
This implies that {$h_\mu(\Phi)=h(\Phi,G_\mu(\Phi))$}, which finishes the proof. $\quad \quad \square$

\begin{remark}\label{r-teoCnvale} \textnormal{(a)} If the time-1 map of a flow is saturated {(and  hence so is the flow)}, then for every set $Y$ such that $G_\mu(\phi^1)\subseteq Y\subseteq G_\mu(\Phi)$ it holds
$$
h_\mu(\phi^1)=h(\phi^1,Y)=h(\Phi,Y).
$$
\indent \textnormal{(b)} It is not hard to give examples such that the flow $\Phi$ is strongly saturated, but $\phi^1$ is not. One can consider a suspension flow with roof function $\rho\equiv1$ and over a uniquely ergodic base.\\
\indent \textnormal{(c)} It is not a simple task to detect when a dynamical system is saturated. In the following subsection we shall give some  such examples.
\end{remark}

The example below provides a class of systems that is never saturated.

\begin{ex}\label{p-naoerrante}
Let $\Phi$ be a continuous flow defined on a compact { Riemannian manifold} $M$. Assume that the nonwandering set of $\Phi$ can be written as a disjoint union $\Omega(\Phi)=\Omega_1\cup\Omega_2$, with $\Omega_1$ and $\Omega_2$  closed and $\Phi-$ invariant. Let $\mu_1,\mu_2\in\mathfrak{M}(\Phi)$ with $\supp(\mu_i)\subseteq\Omega_i$, $i=1,2$. Then  any measure defined as  $\mu=\lambda\mu_1+(1-\lambda)\mu_2$, $0<\lambda<1$ does not admit generic points. Moreover, if  $\Phi$  has positive entropy, it is not saturated.
\end{ex}
\begin{proof}
The proof goes by contradiction. Assume that there is
$x\in G_\mu(\Phi)$. Without loss {of generality}, we can suppose that  $\omega(x)\subseteq \Omega_1$. By  {Urysohn's lemma}, we consider a continuous map $\varphi:M\rightarrow\mathbb{R}$ such that
$$
\varphi(y)=\left\{
\begin{array}{rcl}
0 & \mbox{if} & y\in\Omega_1\\
1 & \mbox{if} & y\in\Omega_2.
             \end{array}
   \right.
$$

 \noindent {\bf{Claim.}}\/ $\varphi(\phi^t(x))\rightarrow0$ when $t\rightarrow\infty$.

\noindent {\em{Proof of the Claim.}} Assume that exists $(t_k)_{k\geq1}$ such that $\varphi(\phi^{t_k}(x))\rightarrow r\neq0$ when $k\rightarrow 0$. Taking a subsequence, if necessary, we can assume that  $(\phi^{t_k}(x))_{k\geq1}$ converges, say to $z_x$. Since $z_x\in\Omega_1$ and $\varphi$ is continuous,  we get
$$
r=\lim_{k\rightarrow\infty}\varphi(\phi^{t_k}(x))=\varphi(\lim_{k\rightarrow\infty}\phi^{t_k}(x))=\varphi(z_x)=0. \quad\quad \square
$$

The claim implies that
$$
\lim_{T\rightarrow\infty}\frac{1}{T}\int_0^T\varphi(\phi^t(x))dt=0,
$$
and since 
$$
\int_M\varphi d\mu\geq\int_{\Omega_2}\varphi d\mu=(1-\lambda)>0,
$$
we arrive at a contradiction to the definition of generic point.
Thus  $G_\mu(\Phi)=\emptyset$. 

Now, if the flow has positive entropy, we can choose $\mu_1$ with $h_{\mu_1}(\Phi)>0$ and so $h_\mu(\Phi)>0=h(\Phi,G_\mu(\Phi))$,  
proving that $\Phi$ is not saturated.
 \end{proof}

The above example includes all the non-transitive {Anosov flows}. In particular, by Theorem \ref{teoC}, the homeomorphism $\phi^1$ is also not saturated and  so $h(\phi,G_\mu(\phi^1))<h_\mu(\phi^1)$ for some $\phi^1$-invariant measure $\mu$. In contrast, we  see below that {topologically mixing} Anosov flows are strongly saturated.

\subsection{Almost specification for flows versus saturated systems}\label{sub-almost-sp}

This subsection aims to propose a definition of almost specification for continuous flows defined on compact metric spaces and establish a version for flows of the well known results by Meson-Vericat and  Thompson.\\
\indent In \cite{Pfister-Sullivan(2007)}, C.-E. Pfister and W. Sullivan introduced the notion of  $g$-almost product property to continuous functions.   D. Thompson \cite{Thompson(2012)} made a {slight}  modification in this concept to introduce the definition of almost specification property. 
We point out that the specification property implies the almost specification property, see \cite[Proposition 2.1]{Pfister-Sullivan(2007)} and \cite[p. 5397]{Thompson(2012)}).
 Still, there are many systems that satisfy the almost specification property which do not satisfy the specification property. For instance, {every $\beta$-shift $T_\beta$ has the almost specification property } \cite[Theorem 5.1]{Thompson(2012)}; however, the set of  $\beta\in(0,\infty)$ such that $T_\beta$ satisfies the specification property has zero Lebesgue measure (although it is dense with Hausdorff dimension equal to 1), see \cite[Theorem 1.5]{Buzzi(2005)} and \cite[Theorem A]{Schmeling(1997)}).

 To make the text self contained, we start recalling the notion of almost specification for discrete dynamical. The definition and some remarks on the specification property are outlined in the Appendix.

\subsubsection{Almost specification for discrete dynamical systems}\label{s-almost-discrete}

 We start defining the notion of  {\em{mistake function}}.
\begin{defi}
Let $\varepsilon_0>0$ fixed. A function $g:\mathbb{N}\times(0,\varepsilon_0]\rightarrow\mathbb{N}\cup \{0\}$ is a \textbf{mistake} if for every $\varepsilon\in(0,\varepsilon_0]$ the following holds:\\
\indent (A1) $g(n,\varepsilon)$ is non decreasing respect to $n$.\\
\indent (A2) $\displaystyle{\lim_{n\rightarrow\infty}\frac{g(n,\varepsilon)}{n}=0}$.
\end{defi}
We extend the function $g$ for every $\varepsilon>\varepsilon_0$ defining $g(n,\varepsilon):=g(n,\varepsilon_0)$.\\
\indent Let $f:X\rightarrow X$ {be a continuous map} and
 $\Lambda\subseteq\mathbb{N}\cup \{0\}$ be a non empty finite set.
We define the {pseudo-metric} on  $\Lambda$ induced by $f$ as
$$
d^f_\Lambda(x,y):=\max\{d(f^jx,f^jy):j\in\Lambda\}.
$$
A \textbf{Bowen ball} along $\Lambda$ at $x$ and radius $\varepsilon>0$ is given by
$$
{B^f_\Lambda(x,\varepsilon):=\{y\in X:d^f_\Lambda(x,y)<\varepsilon\}.}
$$
\indent If $n$ is a {positive integer} we denote   $\Lambda_n:=\{0,1,...,n-1\}$.\\
\indent Given a mistake function $g$, $n\in\mathbb{N}$ and $\varepsilon>0$ with $n>g(n,\varepsilon)$ we define
$$
I(g|n,\varepsilon):=\{\Lambda\subseteq\Lambda_n:\#\Lambda\geq n-g(n,\varepsilon)\}
$$
and
$$
B^f_n(g|x,\varepsilon):=\{y\in X:y\in B^f_\Lambda(x,\varepsilon)\mbox{ for some }\Lambda\in I(g|n,\varepsilon)\} =\bigcup_{\Lambda\in I(g|n,\varepsilon)}B^f_\Lambda(x,\varepsilon).
$$
\begin{defi}\label{Almost-defi} A continuous function $f:X\rightarrow X$ satisfies  the property of \textbf{almost {specification}} (ASP for short) if there  exists a mistake function $g$ satisfying the following:
Given real numbers $\varepsilon_1, \varepsilon_2,...,\varepsilon_k>0$, there are {integers}  $N_g(\varepsilon_i)$, $i=1,2,...,k$, such that for any $x_1, x_2,..., x_k\in X$ and  {integers} $n_i\geq N_g(\varepsilon_i)$, $i=1,2,...,k$, there is $z\in X$ such that
$$
{f^{N_{j}}(z)\in B^f_{n_j}(g|x_j,\varepsilon_j)\quad\mbox{for all }\, j=1,2,...,k,}
$$
where $n_0:=0$ {and} $N_j:=n_0+n_1+...+n_{j-1}$.
\end{defi}
It is easy to see  that this property is inherited by factors, invariant by conjugacy and does not depend on the choice of the metric on $X$. 
{
\begin{prop}\label{p-almost} If $f:X\rightarrow X$ is a homeomorphism with the almost specification property, then $f^{-1}$ also has the almost specification property.
\end{prop}
\begin{proof} Suppose $f$ has the almost specification property with mistake function $g$. Let $\varepsilon_1, \varepsilon_2,...,\varepsilon_k>0$ and $x_1, x_2,..., x_k\in X$.  Let $N_g(\varepsilon_i)$, $i=1,2,...,k$ as in Definition \ref{Almost-defi} and let $n_i>N_g(\varepsilon_i)$ be integers. By the almost specification property, there is a point $\overline{z}\in X$ such that 
$$
f^{\overline{N}_{j}}(\overline{z})\in B^f_{n_{k+1-j}}(g|f^{-n_{k+1-j}}(x_{k+1-j}),\varepsilon_{k+1-j})\quad\mbox{for all }\, j=1,2,...,k,
$$
where $n_0=0$ and $\overline{N}_j:=n_k+n_{k-1}+...+n_{k+1-j}$.\\
Now, for $z=f^{n_1+n_2+\cdot+n_k}(\overline{z})$ we see that
$$
f^{-N_{j}}(z)\in B^{f^{-1}}_{n_j}(g|x_j,\varepsilon_j)\quad\mbox{for all }\, j=1,2,...,k,
$$
where $n_0:=0$ and $N_j:=n_0+n_1+...+n_{j-1}$. This finishes the proof.
\end{proof}}

Despite  its rather complicated appearance, the almost specification property 
become a handy tool in topological and ergodic {theory}. 
For instance, it was shown in \cite[Theorem 3.5]{Kulczycki-Kwietniak-Oprocha(2014)} that whenever the system is surjective, this property implies the \textit{asymptotic average shadowing property} (AASP for short) which, in turn implies chain mixing (\cite[Theorem 3.1]{Kulczycki-Oprocha(2011)}). In \cite[Theorem 5.1]{Kulczycki-Kwietniak-Oprocha(2014)}, it was proved that $f$ has the almost specification property provided $f$ restricted to its \textit{measure center} (the closure of the union of all supports of f-invariant measures) has that property. On the other hand, it is not hard to see that the almost specification property implies the \text{approximate product property} introduced by Pfister and Sullivan in \cite{Pfister-Sullivan(2005)}, and  {the space of ergodic measures} of a system satisfying the almost specification property is entropy dense {(\cite[Theorem 2.1]{Pfister-Sullivan(2005)})}. Besides, \cite[Corollary 22]{Kwietniak-Lacka-Oprocha(2017)}
(see also \cite{Dong-Tian-Yuan(2015)}) allows one to conclude that if $f$ has the almost specification property, then $G_\mu(f)$ is nonempty for every invariant measure $\mu$.
In connection with the theory of dimension, it was proved by Pfister and Sullivan \cite[Theorem 6.1]{Pfister-Sullivan(2007)} that systems satisfying the property of $g$-almost product are saturated and then A. Meson and F. Vericat at \cite{Meson-Vericat} extend this result for systems satisfying the almost specification property. 
In \cite[Theorem 1.1]{Fan-Liao-Peyriere}, the authors prove
 that systems satisfying the specification property are saturated. 
 
 \subsubsection{Almost specification for continuous dynamical systems}\label{s-almost-cont}

\indent {Let $\mathbb{R}^+$ be the set of positive real numbers and let
$\mathbb{R}_0^+$ be the set of nonnegative real numbers.} We denote by $\mathcal{L}(\mathbb{R}_0^+)$ the Lebesgue $\sigma-$algebra and by $\lambda$ the Lebesgue measure.

\begin{defi}
Let $\varepsilon_0>0$ fixed. A function $g:\mathbb{R}^+\times(0,\varepsilon_0]\rightarrow\mathbb{R}^+$ is called {a} {\bf mistake function} if for all $\varepsilon\in(0,\varepsilon_0]$ it holds:
\begin{enumerate}
\item[(A1)] $g(\cdot,\varepsilon)$ is continuous and non-decreasing.
\item[(A2)] $\displaystyle{\lim_{t\rightarrow\infty}\frac{g(t,\varepsilon)}{t}=0}$.
\end{enumerate}
\end{defi}
For $\varepsilon>\varepsilon_0$ define $g(t,\varepsilon)=g(t,\varepsilon_0)$ for all $t\in\mathbb{R}^+$.
Let {$\Lambda\subseteq \Lambda_T :=[0,T]$ be} a Lebesgue set. Put
$$
d^{\Phi}_\Lambda(x,y)=\sup\{d(\phi^tx,\phi^ty):t\in \Lambda\}\qquad\mbox{and}\qquad B^{\Phi}_\Lambda(x,\varepsilon)=\{y\in X:d^\Phi_\Lambda(x,y)<\varepsilon\}.
$$
\indent If $g$ is a mistake function, we define
\begin{eqnarray*}
B^{\Phi}_T(g|x,\varepsilon) & = & \{y\in X:\mbox{ there is } \Lambda\in\mathcal{L}(\Lambda_T)\mbox{ with } \lambda(\Lambda_T\setminus \Lambda)\leq g(T,\varepsilon)\mbox{ and } d^{\Phi}_\Lambda(x,y)<\varepsilon \}\\
& = & \bigcup\{B^\Phi_\Lambda(x,\varepsilon):\Lambda\mbox{ is a }\lambda-\mbox{measurable set with } \lambda(\Lambda_T\setminus \Lambda)\leq g(T,\varepsilon)\}.
\end{eqnarray*}

\begin{defi}\label{Defi-ASP} A flow $\Phi:X\times\mathbb{R}\rightarrow X$ has the {\bf almost specification property} (ASP for short) if there is a mistake function $g$ such that for any {positive integer} $k$, any real numbers $\varepsilon_1, \varepsilon_2,...,\varepsilon_k>0$, there are
positive numbers
$T_g(\varepsilon_1), T_g(\varepsilon_2),..., T_g(\varepsilon_k)$ such that for each $x_1, x_2,...,x_k\in X$ and times $t_i\geq T_g(\varepsilon_i)$, $i=1,2,...,k$, there exists $z\in X$ satisfying
$$
\phi^{T_{j}}(z)\in B^{\Phi}_{t_j}(g|x_j,\varepsilon_j)\qquad\mbox{for all }\quad j=1,2,...,k
$$
where $t_0=0$ and $T_j=t_0+t_1+t_2+...+t_{j-1}$.
\end{defi}

It is easy to see that this definition is invariant by conjugation,  so it does not depend on the choice of metric $d$ on $X$. 
We look at its discrete systems $\phi^t$ to establish  interesting results for 
{a flow $\Phi$} with the almost specification property
%To establish some interesting results for {a flow $\Phi$} with the almost specification property, we look at its discrete systems $\phi^t$. 

\noindent {\bf{Proof of Theorem \ref{teoD}.}}
{By Proposition \ref{p-almost} it is enough  to prove for $t>0$, 
and there is no loss of generality {to} consider $t=1$.}
%It is suffices to prove for $t=1$.
{Given $\epsilon > 0$ we pick a fixed $\bar\epsilon >0 $ such that $d(\phi^{t}x, \phi^{t}y) < \epsilon$ for all $x,\, y \in X$ and for every $t \in [0,1]$ whenever $d(x,y) < \bar\epsilon$.}
 Assume that $\Phi=\{\phi^t\}_{t\in\R}$ has ASP with mistake function $\overline{g}:\R^+\times\R^+\rightarrow\R_0^+$. Consider the mistake function $g:\N\times\N\rightarrow\N_0$ defined by $g(n,\varepsilon)=[\overline{g}(n-1,\overline{\varepsilon})]+2$ for any $ n> 1$, $g(1,\epsilon)=g(2,\epsilon)$.
We shall prove that $\phi^1$ has  ASP with mistake function $g$.\\
\indent Let $\varepsilon_1,\varepsilon_2,...,\varepsilon_k>0$ and consider $N_g(\varepsilon_i)=[T_{\overline{g}}(\overline{\varepsilon}_i)]+2$ for $i=1,2,...,k$, where
$T_{\overline{g}}(\overline{\varepsilon}_i)$ is as {in} Definition \ref{Defi-ASP}.\\
\indent Let $x_1,x_2,...,x_k\in X$ and $n_i>N_g(\varepsilon_i)$, $i=1,2,...,k$ {be} integers. Since $\Phi$ has ASP, there is $z\in X$ such that

%\begin{equation}\label{ASP-F2}
$$
\phi^{T_j}(z)\in B^{\Phi}_{n_j}(\overline{g}|x_j,\overline{\varepsilon}_j)\quad\mbox{for  }j=1,2,...,k,
$$
%\end{equation}
where $n_0=0$ and $T_j=n_0+n_1+n_2+...+n_{j-1}$.\\
\indent For $j=1$, let $\Lambda\subseteq[0,n_1]$ be a Borel set satisfying $\lambda([0,n_1]\setminus\Lambda)\leq \overline{g}(n_1,\overline{\varepsilon}_1)$ and $d_\Lambda^\Phi(x,z)<\overline{\varepsilon}_1$. Denote  {$\Omega_i=(i-1,i]\cap\Lambda^c$}, $i=1,2,...,n_1-1$ and set
$$
\Omega=\{i\in\{0,1,...,n_1-1\}:d(\phi^{i}x_1,\phi^{i}z)>\varepsilon_1\}.
$$
Clearly, if $i>0$ is such that $i \in \Omega$, then $\Omega_i=(i-1,i]$, and hence
\begin{eqnarray*}
\#{\Omega}\leq\#\{i\in\{1,2,...,n_1-1\}:\Omega_i=(i-1,i]\}+1 & \leq & [\lambda([0,n_1-1]\setminus\Lambda)]+2\\
& \leq &  [\overline{g}(\overline{\varepsilon_1},n_1-1)]+2\\
& = & g(\varepsilon_1,n_1).
\end{eqnarray*}
Thus
$
z\in B^{\phi^1}_{n_1}(g|x_1\varepsilon_1).
$
The same argument works for $j=2,3,...,k$ and thus
$$
\phi^{T_j}(z)\in B^{\phi^1}_{n_j}(g|x_j,\varepsilon_j)
$$
where $n_0=0$ and $T_j=n_0+n_1+n_2+...+n_{j-1}$, $j=1,2,...,k$, proving that $\phi^1$ has ASP with mistake function $g$.\\
\indent The converse is similar.
Assuma that   $\phi^1:X\rightarrow X$ has ASP with mistake function $\overline{g}$.
We shall prove that the flow  $\Phi$ has ASP with mistake function
 defined by
$$
g(t,\varepsilon)=(\overline{g}(n+1,\overline{\varepsilon})-\overline{g}(n,\overline{\varepsilon}))(t-n+1)+\overline{g}(n,\overline{\varepsilon})+2
$$
where $n-1< t\leq n$, $n\geq1$.\\
\indent To this end, let $\varepsilon_1,\varepsilon_2,...,\varepsilon_k>0$ and $T_g(\varepsilon_j)=N_{\overline{g}}(\overline{\varepsilon}_j)+2$, $j=1,2,...,k$.
Let $x_1, x_2,...,x_k\in X$ {and} $t_1>T_g(\varepsilon_1), t_2>T_g(\varepsilon_2),...,t_k>T_g(\varepsilon_k)$ {be} positive real numbers.\\
\indent Put $n_1=[t_1]+1$ and $n_j=[t_j-1+(t_{j-1}-[t_{j-1}])]+1$ for $j\geq2$.
Clearly $n_j>N_{\overline{g}}(\overline{\varepsilon}_j)$ for all $j=1,2,...,k$.
Put also  $\overline{x}_1=x_1$ and $\overline{x}_j=\phi^{1-(t_{j-1}-[t_{j-1}])}(x_j)$, $j\geq2$.\\
\indent Since  $\phi^1$ has ASP, there exists $z\in X$ such that
$$
\phi^{N_j}(z)\in B^{\phi^1}_{n_j}(\overline{g}|\overline{x}_j,\overline{\varepsilon}_j)
$$
where $n_0=0$ and $N_j=n_0+n_1+n_2+...+n_{j-1}$, $j=1,2,...,k$.\\
{\indent Let $\overline{\Lambda}_j\subseteq\{0,1,...,n_j-1\}$  be a set satisfying $\#(\{0,1,...,n_j-1\}\setminus\overline{\Lambda}_j)\leq \overline{g}(n_j,\overline{\varepsilon}_j)$ and $d_{\overline{\Lambda}_j}^{\phi^1}(x,z)<\overline{\varepsilon}_j$.}

From the choice of  $\overline{\varepsilon}_j$ it follows that if $i\in\overline{\Lambda}_j$, $i\neq n_j-1$, then $d(\phi^{N_j+i+t}z,\phi^{i+t}\overline{x}_j)<\varepsilon$ for every $t\in[0,1]$.\\
\indent Letting  $\Lambda_j=\cup\{[i,i+1]:i\in\overline{\Lambda}_j, i\neq n_j-1\}$ we have that
\begin{eqnarray*}
\lambda([0,t_j]\setminus\Lambda_j) & \leq & \#\{0,1,...,n_j-1\}\setminus\overline{\Lambda}_j+2\\
& \leq & \overline{g}(n_j,\overline{\varepsilon})\\
& \leq &  g(t_j,\varepsilon).
\end{eqnarray*}
Thus,
$$
\phi^{T_j}(z)\in B^{\Phi}_{n_j}(g|x_j,\varepsilon_j),
$$
where $t_0=0$ and $T_j=n_0+t_1+t_2+...+t_{j-1}$, $j=1,2,...,k$. {This concludes that $\Phi$ has ASP with mistake function $g$ and concludes the proof.} $\square$

 As mentioned before,  the specification property implies the almost specification property for discrete time systems.
 %, see \cite[Proposition 2.1]{Pfister-Sullivan(2007)}  and \cite[p. 5397]{Thompson(2012)}. 
 We deduce the same result  for flows, as a corollary of Theorem \ref{Phi=>phi1-Esp}  in the Appendix. 
 
\begin{cor}\label{Fspecific-ASP} If $(X,\Phi)$ has the specification property,  it has the almost specification property.
\end{cor}
\begin{proof} 
Theorem \ref{Phi=>phi1-Esp} ensures that  if $\Phi$ has the specification property 
then 
$\phi^1$ has this property.
 Hence, by \cite[Proposition 2.1]{Pfister-Sullivan(2007)},   $\phi^1$ has the almost specification property. The result follows from Theorem {\ref{teoD}}.
\end{proof}

It is known that every topologically mixing Anosov flow on a closed manifold has the specification property (see \cite[Theorem 18.3.14]{Katok-Hasselblatt}), so the class of  continuous time systems satisfying the almost specification property includes such flows. In particular, every geodesic flow on a closed manifold with negative curvature is in this class. 
In Example \ref{ASPnotSecific}, we shall see that the converse of Corollary \ref{Fspecific-ASP} does not hold in general.

\subsubsection{Consequences of almost specification property}\label{s-conseq}

Theorem {\ref{teoD}} {allows} us to obtain several consequences of the almost specification property for flows.
In the sequel, we list some of them.
%from discrete case.

\begin{teo}\label{Quase-Spec} Suppose that $(X,\Phi)$ has the almost specification property, then the following statements hold:\\
\indent \textnormal{(a)} $G_\mu(\Phi)\neq\emptyset$ for every $\mu\in\mathfrak{M}(\Phi)$.\\
\indent \textnormal{(b)} $\Phi$ is strongly saturated.\\
\indent \textnormal{(c)} $\EF$ is entropy dense. In particular, $\EF$ is dense in $\MF$.\\
{\indent \textnormal{(c)} $G_\mu(\Phi)$ is dense in the support of $\mu$.}
\end{teo}
\begin{proof} (a) Follows from \cite[Corollary 1.5]{Dong-Tian-Yuan(2015)}, Theorem {\ref{teoD}} and Remark \ref{Gnonempty}.\\
\indent \textnormal{(b)} Follows from Theorems \ref{teoC}, {\ref{teoD}} and \cite{Meson-Vericat}.\\
\indent \textnormal{(c)} It an immediate consequence from Theorem {\ref{teoD}}, Theorem \ref{Dense-Entropy} and \cite[Theorem 2.1]{Pfister-Sullivan(2005)}.\\
{\indent \textnormal{(d)} By theorem {\ref{teoD}}, the time-1 map has the almost specification property and  by \cite[Corollary 1.5]{Dong-Tian-Yuan(2015)}, $G_\mu(\phi^1)$ is dense in the support of $\mu$. The conclusion follows from Proposition \ref{conten}.}
\end{proof}

We note that item (b) from the theorem above is a version for flows of \cite{Meson-Vericat}. 

As an application of Theorem \ref{Quase-Spec},  we have the following

\begin{cor} Every topologically mixing Anosov flow on a closed manifold is strongly saturated. In particular, geodesic flows on a closed manifold of negative curvature are strongly saturated.
\end{cor}

In connection with topological dynamics, it is known that almost specification property implies chain mixing for surjective systems (see, for instance, \cite[Lemma 3.2]{Kulczycki-Kwietniak-Oprocha(2014)}). 
The same is true for continuous dynamical systems, as we shall prove below.
%The same is true for continuous dynamical systems, {as we shall prove below.}

%For $\varepsilon>0, T>0$, a finite sequence $\{x_i\}_{i=0}^n$ is called an {\em $(\varepsilon, T)$-chain} if there exists $\{t_i\}_{i=0}^{n-1}$ such that $t_i>T$ and $d(\phi_{t_i}(x_i),x_{i+1})<\varepsilon$  for all $i=0,\ldots, n-1$. We say that $y$ is {\em chain attainable from $x$}, if there exists $T>0$ such that for all $\varepsilon>0$, there exists an $(\varepsilon, T)$-chain $\{x_i\}_{i=0}^n$ with $x_0=x$ and $x_n=y$.

Let $\delta,T>0$ and $x,y\in X$. A sequence $\{(x_i,t_i)\}_{i=0}^{n}$ in $X\times\mathbb{R}^+$ is is a $(\delta,T)$-{\bf chain} (or $(\delta,T)$-{\bf pseudo orbit}) from $x$ to $y$, if $x_0=x$,  $x_{n+1}=y$,  $t_i\geq T$  and $d(\phi^{t_i}(x_i),x_{i+1})<\delta$ for all $0\leq i\leq n$. In this case, $n$ is the {\it length} of the chain.

\begin{defi} The system $(X,\Phi)$ is {\bf chain mixing} if, for any $x,y\in X$ and any $\delta,T>0$, there is a positive integer $N=N(x,y, \delta,T)$ such that for each $n\geq N$, there exists a $(\delta,T)$-chain of length $n$ from $x$ to $y$. 
%If  $N$ does not depend on $x$ and $y$, the system is called {\bf uniformly chain mixing}.
\end{defi}

\begin{remark}\label{Chains} Concerning the topological notions of mixing, {transitivity} and chain mixing,
we have that the following relations hold to discrete as well to continuous dynamical systems:
\begin{enumerate}
\item[(a)] Topological transitivity implies chain transitivity,
\item[(b)] Topological mixing implies chain mixing,
%\item[(c)] Uniformly chain mixing implies chain mixing.
\end{enumerate}
\end{remark}

%{
%\begin{remark}\label{Chains} The following statements hold to continuous as well discrete  dynamical systems:\\
%\indent (a) Topological transitivity implies chain transitivity.\\
%\indent (b) Topological mixing implies chain mixing.\\
%\indent (c) Chain mixing implies topological mixing.
%\end{remark}}

\begin{prop}\label{ChainMixing} If $\Phi$ has the almost specification property,  it is 
chain mixing.
\end{prop}
\begin{proof} Let $g$ be a mistake function for $\Phi$.
{ Given  $x, y \in X$,  fix  $\delta, T>0$,   $N=4$ and let $n \geq N$. }
Let $T_g(\delta)$ as in {Definition \ref{Defi-ASP}}.  Pick $m\in\N$ such that $m_0=3nmT\geq T_g(\delta)$ and $g(m_0,\delta)/m_0<1/3$. Then $g(m_0,\delta)<nmT$ and hence there exist $nmT<\overline{t}_1,\overline{t}_2<2nmT$ and $z\in X$ such that $d(\phi^{\overline{t}_1}(x),\phi^{\overline{t}_1}(z))<\delta$ and $d(\phi^{\overline{t}_2}(\phi^{-m_0}y),\phi^{\overline{t}_2+m_0}(z))<\delta$. Note that $M=m_0+\overline{t}_2-\overline{t}_1>2nmT>(n-2)T$, so $\ell=M/(n-2)>T$. {Define}
$$
\begin{array}{rcl}
(x_0,t_0) & = & (x,\overline{t}_1)\\
(x_1,t_1) & = & (\phi^{\overline{t}_1}(z),\ell)\\
\vdots & & \vdots\\
(x_{n-1},t_{n-1}) & = & (\phi^{m_0+\overline{t}_2-\ell}(z),\ell)\\
(x_n,t_n) & = & (\phi^{-(m_0-\overline{t}_2)}(y),m_0-\overline{t}_2).
\end{array}
$$
By construction, $\{(x_i,t_i)\}_{i=0}^n$ is the desired $(\delta,T)$-chain from $x$ to $y$ of length $n$.
\end{proof}
\begin{remark} The constant $N=4$ is {independent of $x$, $y$, $\delta$, and $T$.}
\end{remark}

\subsection{Examples}

A usual way to construct examples (or counterexamples) is {through} the
 suspension flows. {Given a homeomorphism
$f:X\rightarrow X$ and a continuous function $\rho:X\rightarrow(0,\infty)$, we recall the definition of the suspension flow of $f$ with roof function $\rho$. To this end, let}
$$
X_{f,\rho}:=\{(x,t)\in X\times\mathbb{R}:0\leq t\leq\rho(x)\}
$$
where, for all  $x\in X$, the points $(x,\rho(x))$ and $(f(x),0)$ are identified.
As in \cite[p. 186]{Bowen-Walters(1972)}, it is possible to induce a topology and a metric
$d_{f,\rho}$ on $X_{f,\rho}$, and as $(X,d)$ is compact, so is $(X_{f,\rho},d_{f,\rho})$.
The space $(X_{f,\rho},d_{f,\rho})$ is called {\em{a suspension space}} of $X$ with roof function
$\rho$, and the metric $d_{f,\rho}$ sometimes is called  \textbf{ Bowen-Walters metric}.
The suspension flow $\Phi_f:X_{f,\rho}\times\mathbb{R}\rightarrow X_{f,\rho}$ of $f$
{with roof function $\rho$} is defined by
$$
\phi_{f,\rho}^t(x,s):=(f^nx,s+t-\rho^n(x))
$$
where $n$ is the unique integer such that $\rho^n(x)\leq s+t<\rho^{n+1}(x)$ and $\rho^n(x)$
is defined recursively by
$\rho^0(x)=0$ and $\rho^{n+1}(x)=\rho^n(x)+\rho(f^n(x))$, $n\in\mathbb{Z}$.\\
When $\rho\equiv1$, we denote $\Phi_f:X_f\times\mathbb{R}\rightarrow X_f$ the suspension
{of}  $f$ and by $d_f$ the metric of Bowen-Walters.

\begin{ex}\label{suspension} Let $f:X \to X$ be a homeomorphism defined on a compact metric space $X$. The flow $\Phi_f:X_f\times\mathbb{R}\rightarrow X_f$ does not satisfy the ASP. 
In particular,
the {time-$1$ map of the} suspension flow does not have the ASP.
\end{ex}
\begin{proof}  Let  $g:\mathbb{R}^+_0\times(0,\varepsilon)\rightarrow\mathbb{R}^+_0$
{be} any mistake function.
Let $w_1=(x_1,s_1), w_2=(x_2,s_2)\in X_f$ with $3/4<s_2<1$ and $\varepsilon_1,\varepsilon_2<1/4$. For any $T_1>0$, there is $t_1 >T_1$ such that ${1/4<\lfloor t_1+s_1\rfloor <1/2}$, where $\lfloor t \rfloor$ is the fractionary part of $t$. 
Let  $T_2>0$ be any real number. If $z=(x,t)\in X_f$ is so that $z\in B_{t_1}(g|w_1,\varepsilon_1)$, then $d(\phi_f^t(z),\phi_f^t(w_2))>1/4$ for all $t\in\mathbb{R}$, implying that $z\not\in \phi_f^{-t_1}B_{t_2}(g|w_2,\varepsilon_2)$ for all $t_2 > T_2$, and hence
$$
B_{t_1}(g|w_1,\varepsilon_1)\cap \phi_f^{-t_1}B_{t_2}(g|w_2,\varepsilon_2)=\emptyset.
$$
Consequently, $\Phi_f$ has not ASP. The second assertion is a consequence of Theorem \ref{teoD}.
\end{proof}

As shown in Example \ref{suspension}, any suspension flow with roof function $\rho\equiv1$
and its time-$1$ map  {do not have} the almost specification property. {In contrast,   }  in the following example, we show that {the suspension flow } and its time-$1$ map {are}  chain mixing if the base $f$ is chain mixing. Consequently, the converse implications in \cite[Lemma 3.2]{Kulczycki-Kwietniak-Oprocha(2014)} and Proposition \ref{ChainMixing} {do}  not hold.

{
\begin{lemma}\label{L-delta} Let $f:X\rightarrow X$ be a homeomorphism on a compact metric space. Let $\delta>0$ and $\ell$ be an integer; then there is a $\overline{\delta}=\overline{\delta}(\delta,\ell)>0$ such that any $\overline{\delta}$-chain of length $\ell$ is $\delta$-shadowed by its first element. 
\end{lemma}
\begin{proof} 
We define $0<\delta_1<\cdots<\delta_{\ell+1}$ recursively as follows:
%Let $0<\delta_1<\cdots<\delta_{\ell+1}$ defined recursively as follows: 
put $\delta_{\ell+1}=\delta$, and if $\delta_{j+1}$ is defined by compactness, we take $\delta_j$ such that  if $ d(x,y)<\delta_j$, then $d(f(x),f(y))<\delta_{j+1}/2$.
Taking $\overline{\delta}=\delta_1$, we finish the proof. 
\end{proof}} 

\begin{ex} Let $f:X\rightarrow X$ be a {chain mixing} homeomorphism and consider $(X_f,\Phi_f)$ the suspension flow over $f$ with roof function $\rho\equiv1$. Then\\
\indent (1) $(X_f,\phi^1_f)$ is chain mixing.\\
\indent (2) $(X_f,\Phi_f)$ is chain mixing.
\end{ex}
\begin{proof} {Let $\delta>0$ and $(x,{t_a}), (y,t_b)$ be two points in $X_f$.  Without loss of generality,  we can assume  $0\leq t_b<t_a<1$. Pick $N=N(x,y,\delta)$ {given by the}  chain mixing property of $f$.}\\
\indent To prove (1), choose $M=M(x,y,{\delta})>2/\delta$ and put $k$ the {smallest} integer
 such that $t_a -\delta{k}/2<t_b$. Consider the following $\delta$-chain on $X_f$:
$$
(x,t_a), (f^1(x),t_a-\delta/2 ), (f^2(x),t_a-\delta),\cdots (f^k(x),t_a-k\delta/2).
$$
From the chain mixing of $f$, for any $n\geq N$ there is a $\delta$-chain of length $n$, $\{(f^{i+k}(x),t_b)\}_{i=0}^n$, from $(f^k(x),t_b)$ to $(y,t_b)$. It is clear that
$$
(x,t_a), (f^1(x),t_a-\delta/2 ), (f^2(x),t_a-\delta),\cdots (f^k(x),t_a-k\delta/2), (f^k(x),t_b),..., (f^{n+k}(x),t_b)
$$
is a $\delta$-chain of length $k+n$ from $(x,t_a)$ to $(y,t_b)$. Since $k<M$ and by the hypothesis on $f$, { we conclude that } $(X_f,\phi^1_f)$ is chain mixing taking $N_1(x,y,\delta)=N+M$.

\indent To prove (2), let $T>0$ and fix an integer $\ell\geq T+1$. Let $s>0$ such that $\phi^s_f(x,t_a)=(x_s,t_b)$ for some $x_s\in X$. { Let $\delta>0$ and $\overline{\delta}$ be as in Lemma \ref{L-delta}}. Let $\overline{N}=\overline{N}((x,t_a),(y,t_b),\delta,T)$ be a positive integer with $\overline{N}\ell\geq N(x,y,\overline{\delta})$, where $N(x,y,\overline{\delta})$ is given by the chain mixing's property of $f$. Then there is a $\overline{\delta}$-pseudo orbit (with respect to $\phi^1_f$) of length $\overline{N}\ell$ from $(x_s,t_b)$ to $(y,t_b)$, say $(x_s,t_b)$, $(x_1,t_b)$,..., $(x_{\overline{N}\ell},t_b)$.
By the choice of $\delta, \overline{\delta}$ and $\ell$
$$
(x,t_a,\ell+s), (x_{\ell},t_b,\ell), (x_{\ell+1},t_b,\ell), \cdots (x_{(\overline{N}-1)\ell+1},t_b,\ell)
$$
is a $(\delta,T)$-chain of length $\overline{N}$ (with respect to $\Phi_f$) from $(x,t_a)$ to $(y,t_b)$. For $n\geq \overline{N}$, put $n=\overline{N}+k$ and $s>k\ell$. Then
$$
(x,t_a,\ell), (f^\ell(x),t_a,\ell),\cdots, (f^{(k-1)\ell}(x),t_a,s-(k-1)\ell), (x_{\ell},t_b,\ell), (x_{\ell+1},t_b,\ell), \cdots (x_{(\overline{N}-1)\ell+1},t_b,\ell)
$$
is a $(\delta,T)$-chain of length $n$ (with respect to $\Phi_f$) from $(x,t_a)$ to $(y,t_b)$. Therefore $\Phi_f$ is chain mixing.
\end{proof}

{Note that  }  neither $(X_f,\Phi_f)$ nor $(X_f,\phi^1_f)$ has the almost specification property, {as shown  in Example~\ref{suspension}}.

In Theorem \ref{Fspecific-ASP}, we have seen that specification implies almost specification. The following example shows that the converse is not valid, so the definition of almost specification considered {here}  is more general {than}
 the specification property.

\begin{ex}\label{ASPnotSecific} Consider $X=\mathbb{S}^1$ and define $\Phi(e^{2\pi{i}x},t)=e^{2\pi{ix^{2^{-t}}}}$, where $x\in[0,1]$. The system $(X,\Phi)$ has the almost specification property.
\end{ex}
\begin{proof}
Since $\lim_{|t|\rightarrow\infty}\Phi(e^{2\pi{i}x},t)=1$ for every $e^{2\pi{i}x}\in \mathbb{S}^1$ then for each $\varepsilon>0$ we can take $N(\varepsilon)>0$ such that $d(\phi^t(e^{2\pi{i}\varepsilon}),1)<\varepsilon$ for all $t>N(\varepsilon)$. If we define $g(\varepsilon,t)=N(\varepsilon)$ for all $\varepsilon>0$ and $t>0$, $g$ is a mistake function.\\
\indent We { claim} that the system $(X,\Phi)$ has the almost specification property with mistake function $g$.
In fact,  {since} every $z_1=e^{2\pi{i}x}\in \mathbb{S}^1$, the point $z=1$ satisfies {that} for any $t>0$, the set
$\{s\geq0:d(\phi^s(e^{2\pi{i}x}),1)\geq\varepsilon\}\subseteq [0,N(\varepsilon)]$, so $\lambda\{s\geq0:d(\phi^s(e^{2\pi{i}\varepsilon}),1)>\varepsilon\}\leq N(\varepsilon)=g(\varepsilon, t)$.
\end{proof}

\indent It is not hard to see that the above example does not have  the specification property. 

 To finish this section, we point out that Proposition \ref{ChainMixing} implies that the map described in Example \ref{ASPnotSecific} is chain mixing; it is neither topologically mixing nor  transitive. Thus the converse of Remark \ref{Chains}(a) and  (b) are not valid.
In particular, we obtain that the almost specification property does not imply topologically transitivity
{nor} topologically mixing.

\subsection{Irregular points versus saturated systems}

This section aims to prove a version of  a result of Thompson for flows (\cite[Theorem 4.1]{Thompson(2012)}).

\begin{defi} Let $\varphi\in C(X)$. A point $x\in X$ is $\varphi$-{\bf irregular} to $f:X\to X $ if its  Birkhoff averages {do not} converge.
\end{defi}
Denote  $\hat{X}(f,\varphi)$ the set of all $\varphi-$irregular points to $f$.
 Birkhoff's Ergodic Theorem implies $\mu(\hat{X}(f,\varphi))=0$ for all invariant measures $\mu$. {Despite} this apparent disadvantage, it can be topologically big (see, for instance, \cite{Ercai-Kupper-Lin}). The corresponding notion for a flow is entirely analogous, and the $\varphi$-{\bf irregular} points for a flow $\Phi$ are denoted by $\hat{X}(\Phi,\varphi)$.

\begin{lemma}\label{IF=If} Let $\Phi$ be a continuous flow defined on a compact metric space $X$. Then
$$
\hat{X}(\Phi,\varphi)=\hat{X}(\phi^1,\overline{\varphi})
$$
where $\overline{\varphi}(x):=\int_0^1\varphi(\phi^t(x))dt$.
\end{lemma}
\begin{proof}
It is clear that $\displaystyle{\lim_{T\rightarrow\infty}\frac{1}{T}\int_0^T\varphi(\phi^t(x))dt}$ converges if and only if $\displaystyle{\lim_{n\rightarrow\infty}\frac{1}{n}\sum_{j=0}^{n-1}\overline{\varphi}(\phi^j(x))}$ converges.
\end{proof}

The following result {gives}  that if $\Phi$ is a continuous flow with almost specification property, then $\hat{X}(\Phi,\varphi)$ is empty or carries total entropy, extending a result due to Thompson \cite[Theorem 4.1]{Thompson(2012)}.

\begin{teo}\label{ThompsonFlow} Let $\Phi$ be a continuous flow satisfying the almost specification property and $\varphi:X\rightarrow\mathbb{R}$ be a continuous function. If $\hat{X}(\Phi,\varphi)$ is not empty, then
$$
h(\Phi,\hat{X}(\Phi,\varphi))=h(\Phi).
$$
\end{teo}
\begin{proof} By Theorem {\ref{teoD}}, $\phi^1$  has ASP. If $\hat{X}(\Phi,\varphi)$ is non-empty, then {by} Lemma \ref{IF=If}, $\hat{X}(\phi^1,\overline{\varphi})$ is non-emtpy, where $\overline{\varphi}(x)=\int_0^1\varphi(\phi^t(x))dt$. We obtain that
\begin{eqnarray*}
	h(\Phi) & = & h(\phi^1)\\
	& = & h(\phi^1,\hat X(\phi^1,\overline{\varphi}))\qquad\mbox{{\tiny{by Thompson's Theorem}}}\\
	& = & h(\Phi,\hat X(\phi^1,\overline{\varphi}))\qquad\mbox{{\tiny{by Theorem \ref{teoA}.}}}\\
	& = & h(\Phi,\hat X(\Phi,\varphi))\qquad\mbox{{\tiny{by Lemma \ref{IF=If}}}}.\\
\end{eqnarray*}
The statement of {Theorem \ref{ThompsonFlow} }follows.
\end{proof}
\vspace{0.1cm}

\noindent {\em{Acknowledgments. }} We are grateful to the anonymous referee  for his/her valuable comments and corrections that improved the text.

\appendix
\section{Specification} \label{specification}
%\section{Appendix} \label{specification}
Bowen \cite{Bowen(1971)PeriodicPoints} introduced the property of specification to study the entropy of Axiom A diffeomorphisms.
%The property of specification was introduced by Bowen \cite{Bowen(1971)PeriodicPoints} to study the entropy of Axiom A diffeomorphisms. 
In this work, we shall use a slightly modified definition that is {{ most common nowadays}. Systems satisfying this property enjoy important dynamical properties from  a topological and ergodic point of view. We refer to the interested reader \cite{Kwietniak-Lacka-Oprocha(2016)} and the references therein for more on this.

Let $T>0$ be a real number. A {\it $T$-spaced specification} for a flow $\Phi$, is a family $\xi=\{\phi^{[s_i,t_i]}(x_i)\}_{i=1}^n$ with $s_i-t_{i-1}\geq T$ for all $i=2,3,...,n$. A collection  $\xi$ is {$\varepsilon$-\textbf{traced}} by $z \in X$ if
$$
d(\phi^t(x_i),\phi^t(z))<\varepsilon\qquad\mbox{for all }t\in[s_i,t_i],\qquad i=1,2,...,n.
$$

\begin{defi}\label{DefiEspFlux} A flow has the \textbf{specification property } if, for every $\varepsilon>0$, there is $T=T(\varepsilon)>0$ such that every $T$-spaced specification is $\varepsilon$-traced by some point in $X$.
\end{defi}

%\textcolor{red}{ Since that the specification property  only concerns positive orbits of the flow,
%to guaranty that it covers the negative orbits of the flow 
%In your proof of Theorem D, what you prove is in fact that ? has ASP if and onlyif?t hasASPforanyt>0. Idonotseethereasonwhyitalsoholds for t < 0. }
{
\begin{prop}\label{p-valeinversa} If $f:X\rightarrow X$ is a homeomorphism with the specification property,  $f^{-1}$ also has the specification property.
\end{prop}
\begin{proof}  Let $\varepsilon>0$ and $N(\varepsilon)$ as in the definition of specification property. Let $\xi=\{(f^{-1})^{[a_i,b_i]}(x_i)\}_{i=1}^n$ be a  $N(\varepsilon)$-spaced specification for $f^{-1}$. 
Note that the collection $\overline{\xi}=\{f^{[b_n-b_i,b_n-a_i]}(f^{-b_{n+1-i}}x_i)\}_{i=1}^n$ is a $N(\varepsilon)$-spaced specification for $f$, and by the specification property, it is $\varepsilon$-traced by some point $\overline{z}\in X$. Clearly $\xi=\{(f^{-1})^{[a_i,b_i]}(x_i)\}_{i=1}^n$ is $\varepsilon$-traced by the point $z=f^{b_n}(\overline{z})$. This ends the proof.
\end{proof}}

\begin{teo}\label{Phi=>phi1-Esp}
The flow $\Phi=\{\phi^t\}_{t\in\RR}$ has the specification property if, and only if, $\phi^t$ has this property for all $t\in\mathbb{R}\setminus\{0\}$.
\end{teo}
\begin{proof} {By Proposition \ref{p-valeinversa}, it is enough  to prove for $t>0$, 
and there is no loss of generality {to} consider $t=1$.}
Assume that $\Phi$ {has the specification property}, and let us prove that so does $\phi^1$.
For this, let  $\varepsilon>0$ and $T=T(\varepsilon)$ as in Definition \ref{DefiEspFlux} and $\xi=\{\phi^{[a_i,b_i]}(x_i)\}_{i=1}^n$ {be a $T$-spaced specification} for $\phi^1$.
 Note that integer numbers form the intervals $[a_i,b_i]$ above.
%Note that the intervals $[a_i,b_i]$ above are formed by integer numbers.
Complete them to an interval in $\RR$, adding all $t \in \RR,\,\, a_{i}\leq t \leq b_{i}$.
By specification, the collection $\overline{\xi}=\{\phi^{[a_i,b_i]}(x_i)\}_{i=1}^n$ is  $\varepsilon$-traced by some $z\in X$; i. e.,
$$
d(\phi^t(x_i),\phi^t(z))<\varepsilon\qquad\mbox{for all  }t\in[a_i,b_i], i=1,2,...,n.
$$
In particular,
$$
d(\phi^j(x_i),\phi^j(z))<\varepsilon\qquad\mbox{for all }a_i\leq j\leq b_i, i=1,2,...,n.
$$
This shows that $\xi$ is $\varepsilon$-traced by the point $z$.\\

\indent Reversely, suppose that   $\phi^1$ has the specification property. Given $\varepsilon>0$,
by continuity of the flow and compactness of $X$, there is  $\overline{\varepsilon}>0$ such that $d(\phi^tx,\phi^ty)<\varepsilon$ for all $t\in[0,1]$ if $d(x,y)<\overline{\varepsilon}$.
Put  $T=T(\varepsilon)=N(\overline{\varepsilon})+1$ and let  $\overline{\xi}=\{\phi^{[s_i,t_i]}(x_i)\}_{i=1}^n$ {be a $T$-spaced} specification for the flow.
Consider the specification  $\xi=\{\phi^{[a_i,b_i]}(x_i)\}_{i=1}^n$ with $a_i=[s_i]$ and $b_i=[t_i]$, where $[t]$ denotes the integer part of  $t$. Clearly  $\xi$ is a  $N(\overline{\varepsilon})$-spaced specification to the time-1 map $\phi^{1}$. By hypothesis, there is  $z\in X$ such that
$$
d(f^jx_i,f^jz)<\overline{\varepsilon}\qquad\mbox{for all }a_i\leq j\leq b_i,\qquad i=1,2,...,n.
$$
By choice of $\overline{\varepsilon}$, we have
$$
d(\phi^t(x_i),\phi^t(z))<\varepsilon\qquad\mbox{for all }t\in[s_i,t_i],\qquad i=1,2,...,n.
$$
This shows that the flow also has the specification property.
\end{proof}

\bibliographystyle{plain}

\noindent Maria Jos\'e Pacifico \,\,$\&$\,\, Diego Alonso Sanhueza\\
Instituto de Matem\'atica,
Universidade Federal do Rio de Janeiro.\\
C. P. 68.530, CEP 21.945-970, Rio de Janeiro, RJ.\\
e-mail: {\em{pacifico@im.ufrj.br \quad
sanhueza.diego.a@gmail.com}}
%%%%%%%%%%%%%%%%%%%%%%%%%%%%%%%%%%%
%%%%%%%%%%%%%%%%%%%%%%%%%%%%%%%%%%
\end{document}